\documentclass[reqno,11pt]{article}

\usepackage[a4paper,top=2cm,bottom=2cm,left=3cm,right=3cm,marginparwidth=1.75cm]{geometry}

\usepackage{amsmath,amsfonts,amssymb,graphicx,amsthm,enumerate,url}
\usepackage[noadjust]{cite}
\usepackage{stmaryrd}
\usepackage{comment,paralist}
\usepackage{mathrsfs,booktabs,tabularx}
\usepackage{xifthen,xcolor,tikz,setspace}
\usetikzlibrary{decorations.pathmorphing,patterns,shapes,calc,decorations}
\usetikzlibrary{decorations.pathreplacing}
\usepackage{mathtools}
\usetikzlibrary{positioning}
\usepackage{soul}

\usepackage[colorinlistoftodos]{todonotes}
\usepackage[colorlinks=true]{hyperref}

\numberwithin{equation}{section}

\newcommand{\rank}{\operatorname{rank}}

\renewcommand{\epsilon}{\varepsilon}

\usepackage{color}
\usepackage{hyperref}
\usepackage[capitalize]{cleveref}
\theoremstyle{plain}

\newtheorem{corollary}{Corollary}[section]
\newtheorem{theorem}[corollary]{Theorem}
\newtheorem{lemma}[corollary]{Lemma}
\newtheorem{claim}[corollary]{Claim}

\newtheorem*{conjecture*}{Conjecture}

\newtheorem*{theorem*}{Theorem}
\newtheorem*{lemma*}{Lemma}
\newtheorem*{definition*}{Definition}
\newtheorem*{corollary*}{Corollary}
\theoremstyle{definition}

\newtheorem{definition}[corollary]{Definition}

\theoremstyle{definition}

\newtheorem{remark}[corollary]{Remark}

\newcommand{\E}{\mathbb E}
\newcommand{\T}{\mathbb T}

\renewcommand{\P}{\mathbb P}
\newcommand{\F}{\mathbb F}

\newcommand{\R}{\mathbb R}

\newcommand{\cM}{\mathcal M}

\newcommand{\cF}{\mathcal F}
\newcommand{\cI}{\mathcal I}
\newcommand{\cG}{\mathcal G}

\newcommand{\cO}{\mathcal O}

\newcommand{\cT}{\mathcal T}

\newcommand{\cW}{\mathcal W}

\newcommand{\bv}{\mathbf v}
\newcommand{\bT}{\mathbf T}

\newcommand{\fa}{\mathfrak a}
\newcommand{\fb}{\mathfrak b}

\newcommand{\be}{\begin{eqnarray}} 
\newcommand{\ee}{\end{eqnarray}}

\newcommand{\eps}{\epsilon}

\DeclareMathOperator{\Tr}{Tr}

\DeclareMathOperator{\res}{\upharpoonright}
\DeclareMathOperator{\ima}{Im}
\DeclareMathOperator{\id}{Id}

\title{Local limits of determinantal processes}
\author{Asaf Nachmias\footnote{Department of mathematical sciences, Tel Aviv university. Email: asafnach@tauex.tau.ac.il}$\,\,$ and Yuval Peled\footnote{Einstein institute of mathematics, Hebrew university of Jerusalem. Email: yuval.peled@mail.huji.ac.il}}

\begin{document}

\maketitle

\begin{abstract} Let $H_n$ be the row space of a signed adjacency matrix of a $C_4$-free bipartite bi-regular graph in which one part has degree  $d(n)\to\infty$ and the other part has degree $k+1$ where $k\geq 1$ is a fixed integer. We show that the local limit as $n\to \infty$ of the determinantal process corresponding to the orthogonal projection on $H_n$ is a variant of a Poisson$(k)$ branching process conditioned to survive. 

This setup covers a wide class of determinantal processes such as uniform spanning trees, Kalai's determinantal hypertrees, hyperforests in regular polytopal complexes, discrete Grassmanians and incidence matroids, as long as their degree tends to $\infty$.
\end{abstract}

\section{Introduction}
We study the local limits of determinantal processes. To do this one must have an adjacency relation between the process' elements which general determinantal processes do not posses. Hence we restrict our setting to processes which naturally have such a structure. This restricted setting may seem unfamiliar at a first glance, but in fact it includes many well studied examples such as the uniform spanning tree, Kalai's determinantal hypertrees \cite{Kalai83} and many more (see \cref{sec:examples}). 

In this setting, the determinantal process is a probability measure on subsets of cardinality $r$ of some finite ground set $U$. We think of $U$ as the right sided vertex subset of a bipartite graph $G=(V,U,E)$ where $V\uplus U$ are the left and right vertex subsets, respectively, and $E$ is the edge set. For a random sample $\cT \subset U$ drawn according to this measure, we study the subgraph $G[\cT]$ of $G$ induced by $V\cup \cT$. Assuming that $G$ is $C_4$-free and $(d,k+1)$-bi-regular with fixed $k$ and $d \to \infty$, we prove that the rooted random graph $(G[\cT],o)$, where $o \in V$ is chosen uniformly at random, converges in distribution with respect to the local topology of rooted graphs to a variant of a Poisson$(k)$ branching process, which we denote by $\T_k$. Notably, the limit depends only on $k$.

Our main result (\cref{mainresult}) recovers the main result of \cite{NP22} regarding the local limit uniform spanning trees of $d$-regular graphs and significantly extends the result in \cite{meszaros2021local} on the local limit of Kalai's determinantal hypertrees to include determinantal hyperforests of any high-degree polytopal complex. More importantly, we provide a unified linear algebraic approach to proving such local limits which include a very wide class of determinantal processes, see \cref{sec:examples} below for examples. 


\subsection{Setup and main result}\label{sec:setup} Suppose that $G_n=(V_n,U_n,E_n)$ is a sequence of simple bipartite graphs where $V_n \uplus U_n$ is the vertex set and $E_n$ is the edge set. For notational convenience we frequently drop the subscript $n$ and denote $G=(V,U,E)$ instead of $G_n=(V_n,U_n,E_n)$. We assume that  $|V|=n$ and $|U|=m$ and that $G$ is $C_4$-free (that is, any two vertices $G$ have at most one common neighbor) and that it is a $(d,k+1)$-bi-regular graph, i.e., the degree of $v\in V$ is $d$ and the degree of $u\in U$ is $k+1$. So $m=nd/(k+1)$. We assume that $k\geq 1$ is a fixed integer and that $d \to \infty$ as $n\to \infty$

Let $B$ be a $n\times m$ signed adjacency matrix of $G$, i.e., $B_{u,v}\in \{1,-1\}$ if $uv\in E$ and $B_{u,v}=0$ otherwise. Denote $r=\rank(B) \leq \min (n,m)$. Let $H \subset \R^{m}$ be the row space of $B$ and denote by $P_H:\R^{m} \to H$ the orthogonal projection operator onto $H$ which we treat as an $m\times m$ matrix represented in the standard basis of $\R^{m}$. Given a subset $\cT\subset U$, we follow Lyons' notation \cite{Lyons03}, and denote $P_H \res \cT$ for the submatrix of $P_H$ whose rows and columns are indexed by $\cT$. The determinantal process associated with $H$ is a probability measure $\P^H$ on subsets $\cT \subset U$ of size $|\cT|=r$ given by
\be\label{eq:defph} \P^H(\cT) = \det(P_H \res \cT) \, .
\ee
To see that \eqref{eq:defph} defines a probability measure let $\{\psi_1,\ldots,\psi_r\}$ be an orthonormal basis of $H$ viewed as row vectors and let $C=(\psi_1,\ldots,\psi_r)$ be an $r\times m$ matrix; then applying Cauchy-Binet on the left hand side of the equality $\det(C C^T)=1$ yields that $\sum_{|\cT|=r} \P^H(\cT) =1$. Also, the $r\times r$ matrix $P_H\res\cT$ is nonsingular only when the columns of $B$ indexed by $\cT$ span an $r$-dimensional space. Hence, the measure $\P^H$ is supported on subsets $\cT\subset U$ such that these columns form a basis of the column space of $B$. In particular, if $\cT \subset U$ is such that $\P^H(\cT)>0$, then every $v\in V$ has a neighbor in $\cT$ (in other words, such $\cT$ must be \emph{spanning}).

Given a subset $\cT \subset U$ we denote by $G[\cT]$ the subgraph of $G$ induced by the vertices $V\cup \cT$. Our goal is to study the local limit (see below) of $G[\cT]$ as $n\to\infty$. The limiting object $\T_k$, first studied  by Linial and the second author \cite{linial2019enumeration}, is a variant of a Poisson$(k)$ branching process. Formally $\T_k$ is an infinite random rooted tree that is drawn according multi type branching process with $4$ types: $\fa$-even, $\fa$-odd, $\fb$-even and $\fb$-odd. The symbol odd/even only represents the parity of the distance from the root (intuitively, the even vertices correspond to $V$ and the odd to $U$). The root $\rho$ is of type $\fa$-even, and the distribution of the number of progeny and their types for a given type is given by the following rules, where all the Poisson$(k)$ random variables below are independent:
\begin{itemize}
    \item A particle of type $\fa$-even has 
    $1$ child of type $\fa$-odd and Poisson$(k)$ children of type $\fb$-odd.
    
    \item A particle of type $\fb$-even has  Poisson$(k)$ children of type $\fb$-odd.
        
    \item A particle of type $\fa$-odd has $k$ children all of them of type $\fa$-even.
    
    \item A particle of type $\fb$-odd has $k-1$ children of type $\fa$-even and $1$ child of type $\fb$-even.
    
 \end{itemize}

\begin{theorem} \label{mainresult} Let $k \geq 1$ be an integer, $d(n)\to \infty$ an integer sequence, $G_n=(V_n,U_n,E_n)$ a sequence of simple $C_4$-free bipartite graphs that are $(d(n), k+1)$ bi-regular, $B_n$ a $|V_n|\times |U_n|$ signed adjacency matrix of $G_n$, $\cT_n$ be drawn according to the determinantal probability measure \eqref{eq:defph} associated with the row space of $B_n$, and $o_n\in V_n$ is chosen uniformly at random independently. Then the rooted graphs $(G[\cT_n],o_n)$ converge in distribution with respect to the local topology to $(\T_k,\rho)$.    
\end{theorem}

The only notion left to recall in order to parse \cref{mainresult} is the \emph{local rooted graph topology}, first introduced in the influential paper of Benjamini and Schramm \cite{BS01}. 
We write $(G,o)$ for a graph $G$ with a distinguished vertex $o$, called the root. 
If $G$ is disconnected, then $(G,o)$ denotes the connected component of $o$ in $G$, rooted at $o$.
Let $\cG_\bullet$ be the set of all locally finite connected rooted graphs $(G,o)$ viewed up to root preserving graph isomorphism.
Given a rooted graph $(G,o) \in \cG_\bullet$ and an integer $r \geq 0$ we write $B_G(o,r)$ for the induced graph of $G$ on the vertices of graph distance at most $r$ from $o$; it is also an element of $\cG_\bullet$. We endow $\cG_\bullet$  with a distance: the distance between $(G_1,o_1)$ and $(G_2,o_2)$ is $2^{-r}$ where $r\geq 0$ is the largest integer so that there is a root preserving graph isomorphism between $B_{G_1}(o_1,r)$ and $B_{G_2}(o_2,r)$. This metric endows what is known as the {\bf local topology} on $\cG_\bullet$ and with it $\cG_\bullet$ is a polish space and the usual notions of convergence in distribution and tightness are viable. In particular, the conclusion of \cref{mainresult} is equivalent to the statement that for any $r\geq 0$ the random rooted graphs $B_{G[\cT_n]}(o_n,r)$ converges in distribution to $B_{\T_k}(\rho,r)$, see \cite[Section 1.2]{curienRG}.


M{\'e}sz{\'a}ros \cite[Lemma  1.4]{meszaros2021local} computed a useful explicit expression for the distribution of $B_{\T_k}(\rho,r)$ for every fixed integer $r$. Let $(T,o)$ be a finite rooted tree. The {\bf height} of a vertex in $T$ is its graph distance from $o$ and the height of $T$ is the maximal height of a vertex in $T$. We say that $(T,o)$ is {\bf valid} if its height is even and every vertex at odd height has degree $k+1$. We call vertices of $(T,o)$ \emph{even} or \emph{odd} according to the parity of their height. Let $V(T)$ and $U(T)$ be the sets of even and odd vertices of $(T,o)$, respectively, and let $I(T)\subset V(T)$ be the even vertices of height \emph{strictly} smaller than the height of $T$. We also denote by Aut$(T,o)$ the set of root preserving  automorphisms of $(T,o)$. Lastly, for $K\subset V(T)$ we write $m(K,T)$ for the number of matchings in $T$ in which the vertices $K \cup U(T)$ are saturated. Lemma 1.4 of \cite{meszaros2021local} states that for any valid finite rooted tree $(T,o)$ of height $r$ we have
\be\label{eq:meszaros1.4}
\P(B_{\T_k}(\rho,r) \cong (T,o)) = \frac{e^{-k|I(T)|} (k!)^{|U(T)|} m(I(T),T)}{|\mathrm{Aut}(T,o)|} \, ,
\ee
where $\cong$ is the rooted graph isomorphism relation. In \cref{sec:local} we give an alternate proof for this equality based on the $1$-out model of Linial and the second author \cite{linial2019enumeration}. It is not used in the main results of this paper, but we believe it sheds more light on the distribution described in \eqref{eq:meszaros1.4}. 
Since $\T_k$ is always infinite and locally-finite, \eqref{eq:meszaros1.4} completely described the distribution of $(\T_k,\rho)$.

Due to \eqref{eq:meszaros1.4}, the following theorem is equivalent to (and is a more concrete version of) \cref{mainresult}.  For $v\in V$ and an integer $r\geq 0$, we abbreviate  $B_\cT(v,r)$ for $B_{G[\cT]}(v,r)$. 



\begin{theorem}\label{thm:mainTree} Assume the setup of \cref{mainresult}, let $(T,o)$ be a valid finite rooted tree of height $r$ and let $v_n$ denote a uniform drawn vertex of $V_n$. Then 
\[
    \lim_{n\to\infty}\P(B_{\cT_n}(v_n,r)\cong (T,o)) = \frac{e^{-k|I(T)|} (k!)^{|U(T)|} m(I(T),T)}{|\mathrm{Aut}(T,o)|} \, ,
\]
where $\cong$ is the rooted graph isomorphism relation.
\end{theorem}

We also obtain a stronger quenched version of \cref{thm:mainTree}.
\begin{theorem}\label{thm:quenched} Assume the setup of \cref{mainresult}, let $(T,o)$ be a valid finite rooted tree of height $r$ and let the random variable $Y_n=Y_n(T,o)$ denote the number of vertices $v_n\in V_n$ for which
\[
B_{\cT_n}(v_n,r)\cong (T,o)\,.
\]
Then 
\[
    \frac{Y_n}{|V_n|} \to \frac{e^{-k|I(T)|} (k!)^{|U(T)|} m(I(T),T)}{|\mathrm{Aut}(T,o)|} \, ,
\]
in probability as $n\to\infty$.
\end{theorem}

\subsection{Examples}\label{sec:examples}

We present examples of determinantal processes of interest that fall into the setting of \cref{mainresult}. 

\subsubsection{Uniform spanning trees} \label{example:ust}
It is well-known that the uniform spanning tree (UST) of a connected graph $\Gamma$ is the determinantal process associated with the rows-space of its oriented incidence matrix. To formulate it precisely in our setup, let $V=V(\Gamma)$, $U=E(\Gamma)$, choose an arbitrary orientation of the edges and set $B$ accordingly. That is,
\[
B_{v,e}=\begin{cases}
    1&\mbox{$v$ is the target of $e$,}\\
    -1&\mbox{$v$ is the source of $e$,}\\
    0&v\notin e.
\end{cases}
\]

Then, the rank of $B$ is $|V(\Gamma)|-1$  and $\P^H$ is the uniform measure on the edge-sets of spanning trees of $\Gamma$. 
In ~\cite{NP22}, it is shown that if $\Gamma$ is a $d$-regular graph, and $d\to\infty$, then the local limit of the UST of $\Gamma$ is the Poisson$(1)$ branching process conditioned to survive.  The special case $\Gamma=K_n$ was established by Grimmet~\cite{grimmett1980random}.

Note that this setting falls into our setup with $k=1$, hence \cref{mainresult} applies. In addition, for every $\cT\subset U=E(\Gamma)$, the graph $G[\cT]$ is obtained from the subgraph $(N_G(\cT),\cT)$ of $\Gamma$ by edge-subdivisions  --- replacing each edge with a path of length $2$. Indeed, the old vertices are $N_G(\cT)$ and the new vertices represent $\cT$. 
Similarly, observe that the limit object $\T_1$ is obtained by a similar edge-subdivisions of the infinite path with an unconditional Poisson$(1)$ branching process ``hanging'' on each vertex of the path. This is precisely the Poisson$(1)$ branching process conditioned to survive. Hence, \cref{mainresult} recovers the result of ~\cite{NP22}. In fact, some parts of our proof extend their methods, which rely of effective resistance in electrical networks, to a more general linear-algebraic setting.

\subsubsection{Kalai's hypertrees} \label{sec:hypertrees}
Kalai \cite{Kalai83} extended the well-known Cayley's formula for the number of labelled trees  to $k$-dimensional simplicial complexes with a complete $(k-1)$-dimensional skeleton. He showed that
\begin{equation}\label{eq:kalai}
\sum_{\cT}|H_{k-1}(\cT)|^2 = n^{\binom{n-2}k}\,,    
\end{equation}
where $H_{k-1}$ denote the $(k-1)$-dimensional homology group with integral coefficients, and the summation is over all $n$-vertex $k$-dimensional simplicial complexes $\cT$ satisfying:
\begin{enumerate}[(1)]
    \item $\cT$ has a complete $(k-1)$-dimensional skeleton, i.e., it contains all simplices of dimension less than $k$ spanned by its $n$ vertices,
    \item $H_{k-1}(\cT;\R)=0,$ and
    \item $H_{k}(\cT;\R)=0$.
\end{enumerate}
When $k=1$, these conditions force $\cT$ to be a spanning tree, and \eqref{eq:kalai} recovers Cayley’s tree enumeration formula. For general $k\ge 1$, complexes with properties (1)–(3) are called \emph{$k$-hypertrees}. It can be shown (see \cite{Kalai83}) that each $k$-hypertree contains precisely $\binom{n-1}{k}$ $k$-faces.

Equation~\eqref{eq:kalai} induces a natural probability measure on $k$-hypertrees of the complete $k$-dimensional simplicial complex with $n$ vertices, 
\begin{equation}\label{eq:hypertrees}
\P(\cT) = \frac{|H_{k-1}(\cT)|^2}{n^{\binom{n-2}{k}}}\, .
\end{equation}
This is a determinantal probability measure. Indeed, if $V=\binom{[n]}{k}$ and $U=\binom{[n]}{k+1}$ are the sets of all $(k-1)$-faces and all $k$-faces of the complete complex, respectively, and $B$ is the $V\times U$ matrix form of the $k$-dimensional boundary operator, then the measure $\P^H$ associated with the row-space $H$ of $B$ is supported on $k$-hypertrees and satisfies~\eqref{eq:hypertrees}. This follows directly from \cite{Kalai83}, see also \cite{lyons09}.

Many properties of Kalai's determinantal hypertree have been studied \cite{kahle2022topology,vander2024simplex,meszaros20242,mesz_rsa,meszaros2025homology,meszaros2025bounds,meszaros2025using}. In particular, M\'esz\'aros \cite{meszaros2021local} proved that the random bipartite incidence graph between the $(k-1)$-faces and the $k$-faces of $\cT$, rooted at a uniformly random $(k-1)$-face converges in distribution to  $\T_k$ with respect to the local topology. This result extends the aforementioned Grimmet's local weak limit of the uniform spanning tree of the complete graph $K_n$ to higher dimensions.

In our setup (see \cref{sec:setup}), this incidence bipartite graph is exactly $G[\cT]$ so M\'esz\'aros' result is another special case of \cref{mainresult}. In fact, \cref{mainresult} significantly extends \cite{meszaros2021local} as it applies to general regular high-degree polytopal complexes (and not just the complete simplicial complex), see \cref{cor:hyperforests} below.

\subsubsection{Hyperforests in high-degree  polytopal complexes}
Counterparts of Kalai's enumeration formula for more general simplicial and cell complexes $X$ have been studied extensively (see the survey  of Duval, Klivans and Martin~\cite{DKM} and the references therein). Every such enumeration formula gives rise to a corresponding determinantal measure on subcomplexes of $X$. As we explain below, our result is applicable to this measure provided that the underlying complex $X$ is polytopal, regular and of high-degree.

Let $\ell\geq 1$ be an integer and $X$ be an $\ell$-dimensional cell complex. A subcomplex $\cT \subset X$ is called an \textit{$\ell$-hyperforest} (referred to as a \textit{maximal spanning forest} in~\cite{DKM}) if
\begin{enumerate}[(1)]
  \item $\cT$ contains all the cells of $X$ in dimensions less than $\ell$,
    \item $H_{\ell-1}(\cT;\R)=H_{\ell-1}(X;\R)$, and
    \item $H_{\ell}(\cT;\R)=0$.
\end{enumerate}
In the case $\ell=1$ of graphs, theses conditions force $\cT$ to be a maximal spanning forest, that is, a spanning tree in each of the connected components of  the graph $X$. In addition, assuming condition (1), conditions (2)–(3) are equivalent to the condition that the columns of $\cT$ in the $\ell$-dimensional boundary operator $B$ of $X$ are a basis for the column-space of $B$~\cite{DKM}. Therefore, the determinantal measure associated with the row-space $H$ of $B$ is supported on (top dimensional faces of) $\ell$-hyperforests of $X$. In addition, Lyons showed in~\cite{lyons09} that $\P^H(\cT)$ is proportional to the squared-size of the torsion subgroup of $H_{\ell-1}(\cT)$. We refer to a complex drawn from this measure as the {\em determinantal $\ell$-hyperforest of $X$}. 
The next result follows immediately from \cref{mainresult}. A cell complex is called polytopal if each of its cells is a polytope and the intersection of any two cells is a face of each of them.

\begin{corollary}\label{cor:hyperforests}
Let $\ell, k \geq 1$ be fixed integers, $d(n)\to\infty$ an integer sequence, and $X_n$ a sequence of polytopal complexes of dimension $\ell$ such that
\begin{itemize}
    
    \item the boundary of every $\ell$-cell contains precisely $k+1$ cells of dimension $\ell -1$,
    \item every $(\ell-1)$-cell is contained in precisely $d$ cells of dimension $\ell$, and 
    \item \textit{(``$C_4$-free")} no pair of $(\ell-1)$-cells is contained in more than one $\ell$-cell.
\end{itemize}
Suppose that $\cT_n$ is the determinantal $\ell$-hyperforest of $X_n$, $o_n$ a uniformly drawn random $(\ell-1)$-cell of $\cT_n$, and let $G[\cT_n]$ be the bipartite incidence graph between $(\ell-1)$-cells and $\ell$-cells of $\cT_n$. Then $(G[\cT_n],o_n)$ converges in distribution to $(\T_k,\rho)$ with respect to the local topology.
\end{corollary}
We note that if $X$ is simplicial then $\ell=k$ and ``$C_4$-freeness" is trivial.  We give here two concrete examples --- apart from the complete simplicial complex --- where Kalai-type enumeration had been studied and for which  \cref{cor:hyperforests} applies.
\begin{enumerate}
    \item Take $X_n$ to be the complete balanced $r$-partite (colorful) $\ell$-dimensional simplicial complex $K_{n/r,\ldots,n/r}$, where $r \ge \ell+1$.
Since $X$ is simplicial, we have $k=\ell$ and $d = n(1-k/r)\to\infty$. Adin established a Kalai-type enumeration formula for colorful simplicial complexes in ~\cite{adin1992counting}.

\item Consider the cubical complex $X_n$, given by the $\ell$-skeleton of the $n$-dimensional hypercube $Q_n = [0,1]^n$. In this case, $d=n-\ell+1\to\infty$ and $k=2\ell-1$. A Kalai-type enumeration in this setting is given in \cite{duval2011cellular}.
\end{enumerate}

\newcommand{\rr}{\ell}

\subsubsection{Discrete Grassmanian}

Let $\F_q$ be a finite field of order $q$ and denote by $\mathrm{Gr}_{\rr}(n)$ the set of linear subspaces of dimension $\rr$ in $\F_q^n$. {We fix $q$ and $\rr$} and let $V=\mathrm{Gr}_{\rr}(n)$ and $U=\mathrm{Gr}_{{\rr}+1}(n)$. We consider a bipartite graph with vertex set $V\uplus U$ by putting an edge $(v,u)$ with $v\in V$ and $u\in U$ if and only if $v \subset u$, and set $B_{v,u}=1$, and $B_{v,u}=0$ otherwise. If ${\rr}+1<n/2$, then $B$ has full row rank, see \cite{stanley2013sperner}. Homology groups over finite fields that correspond to these matrices were studied in  \cite{mnukhin2000modular,tessler2023topological}. 
Denote the size of $\mathrm{Gr}_{\rr}(n)$ by $$\binom nq_{\rr}:={ (q^n-1)(q^n-q)(q^n-q^2)\cdots(q^n-q^{{\rr}-1}) \over (q^{\rr}-1)(q^{\rr}-q)(q^{\rr}-q^2)\cdots(q^{\rr}-q^{{\rr}-1}) }\,.$$
We have that $|V|=\binom nq_{\rr},~|U|=\binom nq_{{\rr}+1}$ and $k+1=\binom {{\rr}+1}{\rr}_{\rr}$. In addition
\[
 d(n) = {q^{n-{\rr}}-1 \over q-1} \to \infty\, ,\]
as $n\to \infty$. Indeed, for any fixed $v\in\mathrm{Gr}_{\rr}(n)$, each $u\in\mathrm{Gr}_{\rr+1}(n)$ containing $v$ corresponds to a one dimensional subspace of the quotient space $\F_q^n / v$ which has dimension $n-\rr$.
Note that $G=(V,U,E)$ is $C_4$-free since there is at most one $({\rr}+1)$-dimensional subspace containing two distinct ${\rr}$-dimensional subspaces. 

 We are not aware of a description of the bases $\cT$ of the column-space of $B$, nor of previous work studying the determinant of $B\res\cT$. Nevertheless, \cref{mainresult} obtains the distributional limit of $G[\cT]$ that is sampled proportionally to this determinant with respect to the local topology.




\subsubsection{Incidence matroids}
Let $V=\binom{[n]}{l}$ and $U=\binom{[n]}{r}$ for some fixed $l<r\leq n$, and $B$ be the unsigned incidence matrix (all $+1$). It is known \cite{stanley2013sperner} that $B$ has a full row rank, but we are not aware of a characterization of the bases of the corresponding matroid or their determinants in general.
In the special case $l=1,r=2$, the bases $\cT$ are $n$-edge graphs with no even cycles and only unicyclic or acyclic components, and $\P^H(\cT)$ is proportional to $4^C$, where $C$ is the number of cyclic components in $\cT$~\cite{zaslavsky1982signed}. Our result shows that the local weak limit of such a random graph is the Poisson$(1)$ branching process conditioned to survive.

\subsection{The distribution of $(\T_k,\rho)$}\label{sec:local}
In this section we give an alternate proof that the distribution of $(\T_k,\rho)$ is given by \eqref{eq:meszaros1.4}; the original proof is due to M\'esz\'aros, see \cite[Lemma 1.4]{meszaros2021local}. 
Our argument is related to the motivation for the original definition by Linial and the second author of $\T_k$ in \cite{linial2019enumeration} as the local limit of a random $1$-out $k$-dimensional simplicial complex.

Let $k$ and $d$ be integers, and consider the following infinite tree $\bT=\bT_{k,d}$ rooted at a vertex $\bv$. As before, the vertices of $\bT$ are split into $V(\bT)$ and $U(\bT)$ --- the vertices of even and odd height respectively. In addition, every vertex from $V(\bT)$ has $d$ children, and every vertex from $U(\bT)$ has $k$ children. For a vertex $v\in V(\bT)$, we denote by $N_\bT(v)$ the neighbor set of $v$. We stress that $N_\bT(v)$ include the $d$ children of $v$ and also its parent if $v$ has one. Let each vertex $v\in V(\bT)$ independently {\it draw} a uniform random vertex $u_v\in N_\bT(v)$. Additionally, let $\cO = \{u_v\mid v\in V(\bT)\}$ and $\bT[\cO]$ be the random subgraph of $\bT$ induced by $V(\bT)\cup\cO$, and abbreviate $B_\cO(\bv,r)$ for $B_{\bT[\cO]}(\bv,r)$ for an integer $r\ge 0$.

\begin{lemma}\label{lem:limit_expression}
Fix an integer $k$ and let $(T,o)$ be a  valid rooted tree of height $r\ge 0$. Then,
\[
\P(B_{\T_k}(\rho,r)\cong (T,o)) = 
\lim_{d\to\infty}\P(B_{\cO}(\bv,r)\cong (T,o)) =
\frac{e^{-k|I(T)|} (k!)^{|U(T)|} m(I(T),T)}{|\mathrm{Aut}(T,o)|} \, .
\]
\end{lemma}
\begin{proof}
The first equality follows by viewing the component of the root $\bv$ in $\bT[\cO]$ as a multi type branching process. One readily verifies that, as $d \to \infty$, the progeny number and type distribution converges weakly to those of $\T_k$. Hence, $(\bT[\cO],\bv)$ converges in distribution with respect to the local topology to $(\T_k,\rho)$.

The types in the process are $\fa$-even, $\fb$-even, and $j$-odd for $0\le j\le k$, where the root $\bv$ is of type $\fa$-even. Intuitively, $v\in V(\bT)$ is of type $\fa$-even if $u_v$ is his child and is of type $\fb$-even if $u_v$ is his parent.
Note that we have more odd types here than in $\T_k$ since we need to account for the possibility that a vertex in $U(\bT)$ is drawn by more than one of its neighbors. Namely, a vertex $u \in U(\bT)$ is of type $j$-odd if it was drawn by $j$ of his children. In particular, type $0$-odd corresponds to $\fa$-odd in $\T_k$ and type $1$-odd to $\fb$-odd.

To sample the branching process we use the following source of randomness. For each even particle $v$ we encounter, we sample independent random variables $X_1,...,X_d$ that are binomially distributed with $k$ trials and success probability $1/(d+1)$. One should view $X_i$ as the number of children of $u_i$ --- the $i$-th child of $v$ --- that drew $u_i$. Here is the precise description of the process: 

\begin{itemize}
    \item A particle of type $j$-odd, for $0\le j\le k$, has $j$ children of type $\fb$-even and $k-j$ children of type $\fa$-even. 
    \item A particle of type $\fa$-even has one child of type $X_1$-odd, and $n_j$ children of type $j$-odd, where
    \(
    n_j= |\{2 \le i\le d\mid X_i=j\}|,
    \)
    for $1\le j\le k$. Intuitively, the first child, of type $X_1$-odd, is the one drawn by the $\fa$-even particle.
    \item A particle of type $\fb$-even has $m_j$ children of type $j$-odd, where
    \(
    m_j= |\{1 \le i\le d\mid X_i=j\}|,
    \)
    for $1\le j\le k$.
\end{itemize}
It is clear that this branching process and the component of $\bv$ in $\bT[\cO]$ have the same distribution.  In addition, the convergence to the law of $\T_k$ is also straightforward. Indeed, $X_1,n_2,...,n_k,m_2,...,m_k$ all converge to $0$ in distribution as $d\to\infty$, and $n_1,m_1$ converge to Poisson$(k)$.

To prove the second equality, denote by $c(v)$ the number of children of $v\in V(T)$ in the rooted tree $T$. Then, the number of rooted-graph embeddings of $(T,o)$ into $(\bT,\bv)$ is
\[
\left(\prod_{v \in I(T)} d(d-1)\cdots(d-c(v)+1)\right)\cdot (k!)^{|U(T)|} = (1+O(d^{-1}))\cdot (d\cdot k!)^{|U(T)|}
\,,
\]
where the last equality uses $\sum_{v\in I(T)}c(v)=|U(T)|$. Therefore, there are 
 $$(1+O(d^{-1}))\cdot \frac{(d\cdot k!)^{|U(T)|}}{|\mathrm{Aut}(T,o)|}$$ different rooted subgraphs of $(\bT,\bv)$ that are isomorphic to $(T,o)$. 
 
 Fix such a rooted subgraph $T'$. The event $E$ that $B_{\cO}(\bv,r)=T'$ occurs if and only if (i) $U(T')\subset \cO$ and (ii)  $N_\bT(v)\cap \cO \subset U(T')$ for every $v\in I(T')$. Note that condition (i) is equivalent to $U(T')\subset \{u_v:v\in V(T')\}$ since $T$ is valid.

 Denote by $E'$ the event that there exists a vertex $u\in U(T')$ that is chosen by at least two vertices in $V(T')$. If $E\cap E'$ occurs then there are at least $|U(T)|+1$ vertices in $V(T')$ that chose a vertex in $U(T')$. Therefore,
 \[
 \P(E\cap E') \le \left(|V(T')|\cdot |U(T')|\cdot\frac1d\right)^{|U(T')|+1} =
 O\left(d^{-(|U(T)|+1)}\right)\,.
 \]

In addition, for every matching $M$ in $T'$ in which the vertices $I(T')\cup U(T')$ are saturated, consider the event $E_M\subset E\setminus E'$ in which 
\begin{itemize}
    \item every vertex $v\in V(T')$ that is saturated in $M$ drew its pair in $M$ as $u_v$, 
    \item $N_G(v)\cap \cO \subset U(T')$ for every $v\in I(T')$, and
    \item $u_v\notin U(T')$ for every vertex $v\in V(T')$ that is unsaturated by $M$.
\end{itemize}
 We have that
\begin{align*}
    \P(E_M) =& \frac 1d\cdot \left(\frac 1{d+1}\right)^{|U(T)|-1}\cdot \prod_{v\in I(T')}\left(\frac d{d+1} \right)^{k(d-c(v))} \cdot \left(\frac{d}{d+1}\right)^{|V(T)|-|U(T)|} 
    \\=& (1+O(d^{-1}))\cdot d^{-|U(T)|}\cdot e^{-k|I(T)|}\,.
\end{align*}

The first inequality is obtained as follows. The first term accounts for the choice of the root (which has only $d$ neighbors), the second term corresponds to the choices of the other vertices in $V(T')$ that are saturated in $M$, the third term reflects the requirement that $N_G(v)\cap \cO \subset U(T')$ for every $v\in I(T')$,  and the last term reflects the requirement that the unsaturated vertices in $V(T')$ (which must be leafs) do not choose their parent. 

 We claim that $E\setminus E'$ is a disjoint union of the $m(I(T),T)$ events of the form $E_M$. Indeed, suppose that $E\setminus E'$ occurs, and consider a matching in $T'$ that pairs $(v,u)$ if $u=u_v$. The assumption that $E'$ does not occur guarantees that this is a well-defined matching, and the occurrence of $E$ implies that $I(T')\cup U(T')$ are saturated. Therefore,
 \[
 \P(E) = (m(I(T),T)+O(d^{-1}))\cdot d^{-|U(T)|}\cdot e^{-k|I(T)|}\,.
 \]

In conclusion, by summing over the different rooted subgraphs $T'$ of $(\bT,\bv)$ that are isomorphic to $(T,o)$, we find that 
\[
\P(B_{\cO}(\bv,r)\cong (T,o)) =
 (m(I(T),T)+O(d^{-1}))\cdot\frac{(d\cdot k!)^{|U(T)|}}{|\mathrm{Aut}(T,o)|}\cdot d^{-|U(T)|}\cdot e^{-k|I(T)|}\,,
\]
and the proof is derived directly.
\end{proof}

\subsection{Some basic facts about determinantal probability measures} \label{dec:deter}

We describe here some useful facts about the determinantal probability measure $\P^H$ defined in \eqref{eq:defph}, all can be found in Lyons' seminal paper \cite{Lyons03}.  Even though $\P^H$ description in \eqref{eq:defph} is complete, it is sometimes more convenient and common to describe its finite marginals: for any subset $E \subset U$ we have that 
\be\label{eq:PHmarginal} 
\P^H(E \subset \cT) = \det(P_H \res E) \, .
\ee
We will also frequently use the \emph{negative correlation} property of $\P^H$, that is, 
\be\label{eq:negative}
\P^H(U_1\uplus U_2\subset \cT) \leq \P^H(U_1 \subset \cT)\P^H(U_2 \subset \cT) \, ,
\ee
for every disjoint subsets $U_1,U_2\subset U(G)$,
see \cite[Theorem 6.5]{Lyons03} or \cite[Theorem 7.8.9]{MatrixBook}. Next, if $H' \subset H$ is a linear subspace of $H$, then 
\be \label{eq:dominate}
\P^{H'}  \preccurlyeq \P^H \, ,
\ee
where by $\P^{H'}  \preccurlyeq \P^H$ we mean that $\P^{H}$ stochastically dominates $\P^{H'}$. Lastly, conditioning on the event that some elements belong to $\cT$ and others do not, we get another determinantal probability measure corresponding to a certain linear subspace. In particular, for two disjoint subsets $A,B$ of $U$, it is shown in \cite[(6.5)]{Lyons03} that
\be \label{eq:conditioning} 
\P^H(\cdot \mid A \subset \cT, B\cap \cT = \emptyset) = \P^{((H\cap [A]^\perp)+[A]+[B])\cap [B]^\perp}(\cdot) \, ,
\ee
where by $[A]$ we mean the linear subspace of $\R^m$ having non-zero values only on the coordinates of $A$. 
In words, the subspace corresponding to conditioning that 
$B\cap\cT=\emptyset$ is the projection of $H$ onto the orthogonal complement of $B$.
On the other hand, conditioning that  $A\subset \cT$ is 
obtained by taking the direct sum of $[A]$ with the intersection of $H$ and the orthogonal complement of $[A]$.

\section{Proofs}

We begin with some  definitions and notations central to the proof. Recall that $G=(V,U,E)$ is a bipartite graph that is $(d,k+1)$-bi-regular, that $d\to \infty$ and that $B$ is a $|V|\times |U|$ signed adjacency matrix of $G$. We denote $n=|V|$ and $m=|U|$.

Let $L^+:=\frac1d B B^T \in \R^{n\times n}$ and $L^{-}=\frac1d B^T B \in \R^{m\times m}$ be the two natural ``Laplacian'' operators. Both operators are positive semi-definite and have the same set of positive eigenvalues denoted by  $\lambda_1,...,\lambda_r$ where $r \leq \min(n,m)$ is the rank of $B$. We  write $\psi_1,...,\psi_r \in \R^m$ for a choice of orthogonal unit eigenvectors of $L^-$ corresponding to $\lambda_1,\ldots,\lambda_r$. We fix this choice throughout the rest of the proof. The vectors $\{\psi_1^T,\ldots,\psi_r^T\}$ form an orthonormal basis of the row space $H=\ima(B^T)$ of $B$, hence the orthogonal projection $P_H$ onto the $H$ is
$$ P_H = \sum_{i=1}^r \psi_i \psi_i^T \, .$$



Next we set parameters $\epsilon=\epsilon(n)>0$ and $\delta=\delta(n)>0$ that satisfy
\begin{equation}
    \label{eq:eps_del}
    \epsilon\to 0,\;
    \epsilon d\to \infty,\;
    \delta d\to 0,\;
    \mbox{and }\epsilon\delta d^2\to\infty \,,
\end{equation}
as $n\to \infty$ (for instance, $\eps=d^{-1/2}$ and $\delta = d^{-5/4}$ will do). We will use these parameters throughout the proof. The following definition will be crucial.

\begin{definition} \label{def:struct}
We say that a vertex $u\in U$ is {\bf $(\epsilon,\delta)$-structured}, if
\begin{equation}
    \label{eq:typcial}
    \sum_{i:(\lambda_i-1)^2>\epsilon} \psi_i(u)^2 > \delta \, .
\end{equation}
We denote by $U_{\epsilon,\delta}\subset U$ the set of $(\epsilon,\delta)$-structured vertices.
\end{definition}
Informally, $u$ is structured if there is a non-negligible $\ell_2$ mass on it from the eigenvectors corresponding to eigenvalues far from $1$. Intuitively this captures elements of the determinantal process that have ``structure'' around them. For example, in the UST example (see \cref{example:ust}), edges that disconnect the graph into at least two linear sized components can be shown to be structured. We will show that their number is small, namely $o(m)$ (see \cref{clm:Nepsdel1}).

\subsection{Tightness and local lack of structure}

Recall the definition of $B_\cT(v,r)$ appearing above \cref{thm:mainTree}. The main result in this section is the following.

\begin{theorem}\label{thm:TisTight}
Let $o$ be a uniformly drawn vertex of $V$ that is independent from $\cT$. Then for any integer $r>0$ we have 
\begin{enumerate}
\item $\lim_{t\to\infty} \lim_{n\to\infty} \P(|B_\cT(o,r)|>t) = 0$, \rm{and},
\item $\lim_{n\to\infty} \P( B_\cT(o,r) \mathrm{\ is\ a\ tree})=1$, \rm{and}, 
\item $\lim_{n\to\infty}\P(B_\cT(o,r)\cap U_{\epsilon,\delta}\neq\emptyset)=0.$
\end{enumerate}
\end{theorem}

We first show that most of the $r$ positive eigenvalues of $L^-$ are close to $1$. Let 
$$  I :=\{i\in [r]~:~(\lambda_i- 1)^2\le \epsilon\} \, ,$$
and denote by  $H_I$ be the span of $\{\psi_i : i\in I\}$. 

\begin{claim}\label{clm:Neps1} We have that
$$(1-k(\epsilon d)^{-1})n \leq |I|\leq n \, .$$
\end{claim}
\begin{proof}
The upper bound is obvious since $r\leq \min(n,m)$. For the lower bound we consider the operator $(L^+-\id_n)^2$ where $\id_n$ is the identity operator on $\R^n$.
On the one hand, 
\be\label{eq:Nepsmid}\Tr((L^+-\id_n)^2) = \sum_{i=1}^{r}(\lambda_i-1)^2 + n-r \ge (r-|I|)\epsilon+n-r\ge(n-|I|)\epsilon \, ,\ee
since $r \leq n$. On the other hand, for every $v_1,v_2 \in V$
\[
(L^+-\id_n)_{v_1,v_2}=
\left\{ 
\begin{matrix}
0&v_1=v_2 \, ,\\ \pm\frac1d & N(v_1)\cap N(v_2) \neq \emptyset \, ,  \\
0 & \text{otherwise} \, ,
\end{matrix}
\right.
\]
since $N(v_1) \cap N(v_2) \in \{0,1\}$ because $G$ is $C_4$-free. Due to the same reason, the set of vertices $v_2$ such that $N(v_1) \cap N(v_2)=1$ is of size precisely $kd$. Hence $\Tr((L^+-\id)^2)=kn/d$. Together with \eqref{eq:Nepsmid} the claim follows directly.
\end{proof}

\begin{claim}\label{clm:Nepsdel1} We have that
$$ |U_{\epsilon,\delta}| \le {mk(k+1)\over  d^2\epsilon\delta} \, .$$
\end{claim}
\begin{proof}
The sum $\sum_{i\notin I}\|\psi_i\|^2=r-|I|\le nk(\epsilon d)^{-1}$ by \cref{clm:Neps1}. On the other hand, every $u\in U_{\epsilon,\delta}$ contributes at least $\delta$ to that sum. Hence $|U_{\epsilon,\delta}| \leq nk (d\eps \delta)^{-1}$ and since $n=m(k+1)/d$ the claim follows.
\end{proof}

Let $P_{H_I}$ denote the orthogonal projection operator onto $H_I=$span$\{\psi_i : i\in I\}$, i.e., 
\be\label{eq:phi} P_{H_I} = \sum_{i : (\lambda_i-1)^2 \leq \eps } \psi_i \psi_i^T \, . \ee
We denote the corresponding determinantal process by $\P^{H_I}$ so that
$$ \P^{H_I}(F) = \det(P_{H_I} \res F) \, ,$$
and we write $\cF$ for a sample of $\P^{H_I}$.
{For $v\in V$ and an integer $r\geq 0$, we abbreviate  $B_\cF(v,r)$ for $B_{G[\cF]}(v,r)$, where $G[\cF]$ is the subgraph of $G$ induced by $V\cup \cF$.}

\begin{lemma} \label{lem:FTight} For any fixed integer $r>0$ and any $v\in V$ we have 
$$ \lim_{t\to\infty} \lim_{n\to\infty} \P\Big ( |B_\cF(v,r)|>t \Big ) = 0 \, ,$$
and
$$ \lim_{n\to\infty} \P\Big ( B_\cF(v,r) \mathrm{\ is\ a\ tree} \Big ) = 1 \, .$$
\end{lemma}

\begin{proof} Using \eqref{eq:PHmarginal} and  \eqref{eq:phi} and that $\lambda_i \geq 1-\sqrt{\eps}$ for all $i\in I$ we get that for every $u \in U$
\begin{equation}\label{eq:sigmaF}
\P(u \in \cF) = \sum_{i\in I} \psi_{i,u}^2 \le \frac{1}{1-\sqrt{\epsilon}} \sum_{i=1}^{m}\lambda_i \psi_{i,u}^2 = {L^-_{u,u} \over 1-\sqrt{\eps}} =\frac{k+1}{(1-\sqrt{\eps})d},
\end{equation}
where in the last equality we used that $L^-_{u,u}=(k+1)/d$  since every column of $B$ has precisely $k+1$ nonzero entries all of which are $\pm 1$.


Next observe that that if all the vertex degrees of $B_\cF(v,r)$ bounded above by some $M\geq 2$, then $|B_\cF(v,r)|\leq M^{r+1}$. Hence, if $|B_\cF(v,r)|\geq t$, then there must be a vertex of degree at least $t^{1 \over r+1}$. When $t$ is larger than $(k+1)^{r+1}$, this has to be a vertex in $V$ since the degrees in $G$ of vertices of $U$ are $k+1$. Hence there exists a simple path of length $r' \leq r$ in $B_\cF(v,r)$ from $v$ to some $v_1\in V$ and $v_1$ has degree at least $t'=t^{1 \over r+1}\to \infty$ as $t\to\infty$. Note that $r'$ must be even. The number of $v_1$'s that are at graph distance at most $r'$ in $G$ is at most $(dk)^{r'/2}$. For such a path to be included in $B_\cF(v,r)$ we must have that all its $r'/2$ elements of $U$ are in $\cF$ and additional $t'-1$ neighbors of $v_1$ are in $\cF$. Therefore, by negative correlation \eqref{eq:negative} and \eqref{eq:sigmaF} we obtain
\begin{align}
\P(|B_\cF(v,r)|>t)\le & \sum_{r'\leq r}  
(dk)^{r'/2}\binom{d}{t'-1}\left(\frac{k+1}{(1-\sqrt{\eps})d}\right)^{r'/2+t'-1}
\nonumber
\\ \le& \sum_{r'\leq r} (k+1)^{r'}\left(\frac{e(k+1)}{t'}\right)^{t'-1} (1-\sqrt{\eps})^{-r'/2-t'} \, , \label{eq:B_F}
\end{align}
where we used $\binom{d}{t'-1}\leq (ed/(t'-1))^{t'-1}$. When taking first $n\to \infty$ the term $(1-\sqrt{\eps})^{-r'/2-t'}$ converges to $1$; then taking $t\to\infty$ gives convergence to $0$ since $k$ and $r$ are fixed and $t'\to\infty$ as $t\to\infty$. This proves the first assertion of the lemma. 

{
Next, if $B_{\mathcal F}(v,r)$ contains a cycle then there exists a simple path in $B_{\mathcal F}(v,r)$ from $v$ to some vertex $w$ of length $r'$, and a cycle rooted at $w$ of length $2\ell$, disjoint from the path, such that $r'+\ell\le r$. 
The number of such simple paths in $G$ is at most $(kd)^{\lceil r'/2\rceil}$. Moreover, once $w$ is fixed, the number of such cycles is at most $(1+1/k)(kd)^{\ell-1}$. Indeed, since $G$ is $C_4$-free, the final vertex of the cycle is determined by the preceding ones. The factor $1+1/k$ accounts for the case $w\in U$, in which there are $k+1$ choices for the first neighbor of $w$ on the cycle.
For a given path-cycle pair to be contained in $B_{\mathcal F}(v,r)$, all the $\lceil r'/2\rceil+\ell$ vertices of $U$ appearing in it must belong to $\mathcal F$. Hence, by negative correlation \eqref{eq:negative} and \eqref{eq:sigmaF}, we obtain
\begin{eqnarray*}
\P\Big( B_{\mathcal F}(v,r) \mathrm{\ is\ not\ a\ tree} \Big)
&\leq&
2\sum_{r'+\ell\leq r}
(kd)^{\lceil r'/2\rceil+\ell-1}
\left(\frac{k+1}{(1-\sqrt{\eps})d}\right)^{\lceil r'/2\rceil+\ell}
\\
&\leq&
\frac{2}{d}
\sum_{r'+\ell\leq r}
(k+1)^{r'+2\ell} (1-\sqrt{\eps})^{-\lceil r'/2\rceil-\ell} \longrightarrow 0 \, ,
\end{eqnarray*}
as $n\to \infty$ since $\eps \to 0$ and $d\to \infty$ and $r$ is fixed. This proves the second assertion of the lemma and concludes the proof.
}
\end{proof}

By \eqref{eq:dominate} we may couple so that $\cF \subset \cT$ almost surely, so we may assume we work in this coupling measure $\P$.

\begin{lemma}\label{lem:TvsF} Denote by $W\subset V$ the random set of vertices that incident to either $\cT \setminus \cF$ or to $\cF \cap U_{\eps,\delta}$. Let $o$ be a uniformly drawn vertex of $V$ that is independent from $\cF$ and $\cT$. Then
for any fixed integer $r>0$  we have
\[\lim_{n\to\infty}\P(B_\cF(o,r)\cap W\neq\emptyset)=0 \, .
\]
\end{lemma}
\begin{proof} Let $t\geq 1$ be a fixed integer. For any $o\in V$, if $B_\cF(o,r) \cap W \neq \emptyset$ and $|B_\cF(o,2r)|\leq t$, then there must exist $w\in V$ such that 
$$ w\in W , \quad |B_\cF(w, r)| \leq t , \quad  o \in B_\cF(w,r) \, .$$
Since $o$ is a uniformly drawn vertex independent from $\cF$ and $\cT$ we get
\be\label{eq:TvsFmid} \P (B_\cF(o,r) \cap W \neq \emptyset, |B_\cF(o,2r)|\leq t ) \leq {t \E|W| \over n} \, .\ee
Next we have that 
$$ \E |W| \leq (k+1)\E\big [|\cT| - |\cF| + |\cF \cap U_{\eps,\delta}|\big ] \, .$$
Since $|\cF|=\dim(H_I)=|I|$ and $|\cT|=r\leq n$ we have that $|\cT| - |\cF| \leq (\eps d)^{-1} k n$ by \cref{clm:Neps1}. By \eqref{eq:sigmaF} and  \cref{clm:Nepsdel1} we have
$$ \E |\cF \cap U_{\eps,\delta}| = \sum_{u \in U_{\eps,\delta}} \P(u\in \cF) \leq \frac{k|U_{\eps,\delta}|}{(1 - \sqrt{\eps})d} \leq \frac{nk^2}{(1 - \sqrt{\eps})\eps \delta d^2} \, .$$
We use this to bound the left hand side of \eqref{eq:TvsFmid}. Thus, for any $t \geq 1$ we get that
$$\P(B_\cF(o,r)\cap W\neq\emptyset) \leq \P(|B_\cF(o,2r)| > t) + {kn \over \eps d} + {tk^2 \over (1-\sqrt{\eps})\eps \delta d^2} \, .$$
Hence, by choosing $t$ large enough and using \cref{lem:FTight} we obtain the desired result.
\end{proof}

\begin{proof}[Proof of \cref{thm:TisTight}]
Since $\cF \subset \cT$ we always have that $B_\cF(o,r) \subset B_\cT(o,r)$. This containment is strict only if $B_\cF(o,r)$ contains a vertex $v\in V$ that is incident to a vertex $u\in \cT \setminus \cF$. The latter occurs with probability $o(1)$ by \cref{lem:TvsF}. Hence for any integer $r>0$ 
\be \lim_n \P \Big ( B_\cF(o,r) = B_\cT(o,r) \Big ) = 1 \, ,\label{eq:same_balls}\ee
which together with \cref{lem:FTight} implies the first two assertion of the theorem. The third assertion also follows since $B_\cF(o,r) \cap U_{\eps,\delta}=\emptyset$ with probability $1-o(1)$ by \cref{lem:TvsF}.
\end{proof}

\subsection{Inclusion}

For any subgraph $T$ of $G=(V,U,E)$ we write $V(T)$ and $U(T)$ for the vertices of $T$ belonging to $V$ and $U$, respectively. For such a subgraph and $K\subset V(T)$, recall that we write $m(K,T)$ for the number of matchings in $T$ saturating $K \cup U(T)$.

\begin{lemma}\label{lem:inclusion} Let $(T,o)$ be a valid rooted finite tree which is a subgraph of $G$ and $o\in V$. Assume that $U(T) \cap U_{\eps,\delta} = \emptyset$. Then
\[
\P(U(T) \subset \cT) \le (1+o(1)) d^{-|U(T)|} m(\emptyset, T) \, .
\]
\end{lemma}

\begin{proof}
For notational convenience we write $S=U(T)$. Recall that $\psi_1,...,\psi_r$ are orthogonal unit  eigenvectors of $L^-$  corresponding to the non-zero eigenvalues $\lambda_1,\ldots, \lambda_r$ and denote by $\psi_{i,S}\in\R^{|S|}$ their restriction to the entries indexed by $S$. We have \begin{align} \label{eq:inclusionMid}
\P(S\subset \cT) =&~ \det\left(\sum_{i=1}^{r}\psi_{i,S}\psi_{i,S}^T\right) \nonumber \\
\le&~ \det\left(\frac{1}{1-\sqrt{\epsilon}}\left(\sum_{i=1}^{r} \lambda_i \psi_{i,S}\psi_{i,S}^T + \sum_{i\notin I}\psi_{i,S}\psi_{i,S}^T\right)\right) \, ,  \end{align}
where we used the fact that $\lambda_i \geq 1-\sqrt{\eps}$ for all $i\in I$ and $\lambda_i \geq 0$ for all $i\not \in I$ together with the inequality $\det(A+B) \geq \det(A)$ whenever $A$ and $B$ are two positive definite matrices which follows by Minkowski's determinant inequality, see \cite[Theorem 7.8.21]{MatrixBook}.
We also have that $\sum_{i=1}^{r}\lambda_i \psi_{i,S}\psi_{i,S}^T=L^-\res S$, hence the Cauchy-Binet formula implies that
\be\label{eq:inclusionMainTerm}
\det \Big (\sum_{i=1}^{r}\lambda_i \psi_{i,S}\psi_{i,S}^T \Big ) = \det(L^- \res S) = d^{-|S|}\sum_{\substack{Q\subset V\\|Q|=|S|}}\det(B_{Q,S})^2 =d^{-|S|} m(\emptyset, T) \, ,
\ee
since $\det(B_{Q,S})\in \{-1,1\}$ if and only if there is a matching in $T$ saturating $Q \cup S$, and otherwise $\det(B_{Q,S})=0$. Such a matching must be counted in $m(\emptyset, T)$ and conversely each matching in $m(\emptyset, T)$ corresponds to such $Q$. Next, since $L^-\res S$ is invertible the generalized matrix determinant lemma \cite[Theorem 18.1.1]{MatrixBook2} states that 
$$ \det(L^-\res S + CC^T) = \det(\id_\ell + C^T (L^-\res S)^{-1} C) \det(L^-\res S) \, ,$$ 
where $C$ is any $|S|\times \ell$ matrix with any $\ell\geq 1$. We apply this with $C$ the $|S|\times |I^c|$ matrix whose columns are the vectors $\psi_{i,S}\in \R^S$ for $i \in I^c$ so that $CC^T =\sum_{i\notin I}\psi_{i,S}\psi_{i,S}^T$. This together with \eqref{eq:inclusionMid} and \eqref{eq:inclusionMainTerm} yields
\be\label{eq:inclusionMid2}
\P(S\subset \cT) \le \left(\frac{1}{1-\sqrt{\epsilon}}\right)^{|S|}d^{-|S|} m(\emptyset, T)\det(\id_{|I^c|}+C^T(L^-\res S)^{-1}C) \, .
\ee
Since $\eps \to 0$, it suffices to show that that the determinant on the right hand side is $1+o(1)$. For any square matrix $Q$ for which all its eigenvalues are real and at least $-1$ we may bound $\det(\id + Q)\leq \exp(\Tr(Q))$ since $\prod_{i} (1+x_i) \leq e^{\sum_i x_i}$ for any numbers $x_i \geq -1$. Hence,
\be\label{eq:detest}\det(\id_{|I^c|}+C^T(L^-\res S)^{-1}C) \leq \exp(\Tr(C^T(L^-\res S)^{-1}C))=\exp(\Tr(CC^T (L^-\res S)^{-1})) \, ,\ee
where $\Tr(AB)=\Tr(BA)$ implies the second inequality. We expand 
\be\label{eq:trace}\Tr(CC^T(L^-\res S)^{-1})=\sum_{u_1,u_2\in S}(CC^T)_{u_1,u_2}((L^-\res S)^{-1})_{u_2,u_1} \, .\ee
We estimate each term in the sum above. Firstly, by Cauchy-Schwarz
\[
|(CC^T)_{u_1,u_2}|=\left|\sum_{i\notin I}\psi_i(u_1)\psi_i(u_2)\right|\le
\sqrt{\sum_{i\notin I} \psi_i^2(u_1) \sum_{i\notin I}\psi_i^2(u_2)}\le \delta
\]
since $u_1,u_2 \in S$ and $S\cap U_{\epsilon,\delta}=\emptyset$. Secondly, let $M_{u_1,u_2}$ be the corresponding minor of $L^-\res S$. By Cauchy-Binet, 
\[|M_{u_1,u_2}|=\Big |\det(L^-\res_{S\setminus\{u_1\},S\setminus\{u_2\}} )\Big | \le \Big ( \frac{k+1}{d} \Big )^{|S|-1}
\]
since there at most $(k+1)^{|S|-1}$ subsets of $V(T)$ of size $|S|-1$ that can be matched with both $S\setminus \{u_1\}$ and $S\setminus \{u_2\}$. Therefore, using Cramer's rule and \eqref{eq:inclusionMainTerm}, we have
\[
|((L^-\res S)^{-1})_{u_2,u_1}|=\frac{|M_{u_1,u_2}|}{\det((L^-\res S)^{-1})}\le 
\frac{(k+1)^{|S|-1}d}{ m(\emptyset, T)} \, .
\]
Putting these two in \eqref{eq:trace} then back into \eqref{eq:detest} yields

\[
\det(\id_{|I^c|}+C^T(L^-\res S)^{-1}C) \leq \exp\left(|S|^2\frac{(k+1)^{|S|-1}\delta d}{ m(\emptyset,T)}\right) = 1+ o(1) \, ,
\]
since $\delta d\to 0$ by \eqref{eq:eps_del} and all other quantities on the right hand side are fixed.     
\end{proof}

\subsection{Exclusion}\label{sec:exclusion}
Recall that given a valid finite rooted tree $(T,o)$ the set $I(T)\subset V(T)$ is the subset of even vertices (i.e., of $V(T))$ of height \emph{strictly} smaller than the height of $T$.
\begin{lemma}\label{lem:exactly}
Let $(T,o)$ be a valid rooted finite tree of height $r \ge 0$ which is a subgraph of $G$ and $o\in V$. Suppose also that $m(I(T),T)>0$. Then
\[
\P(B_\cT (o,r)= (T,o)  \mid U(T) \subset \cT) \le (1+o(1)) \frac{e^{-k |I(T)|}m(I(T),T)}{m(\emptyset, T)}\, ,
\]
where the $o(1)$ depends on $|V(T)|$.
\end{lemma}

\begin{lemma}\label{lem:no_matching}
Fix an even integer $r \ge 0$. Let $\cT \subset U(G)$ and $o \in V(G)$ be such that $\P^H(\cT) > 0$ and $B_\cT(o,r)$ is a valid rooted tree $(T,o)$ with $|V(T)| < d$. Then the height of $T$ is $r$ and $m(I(T), T) > 0$.
\end{lemma}

These two lemmas, together with \cref{lem:inclusion} give a us an explicit upper bound for the probability $\P(B_\cT (o,r)= (T,o))$ for any fixed valid finite rooted tree $(T,o)$ that is a subgraph of $G$. This brings us very close to \cref{thm:mainTree}. For the proofs of the two lemmas we require some preparations. 

\subsubsection{Transversals in valid trees}
Let $(T,o)$ be a valid finite rooted tree that is a subgraph of $G$ and write $D$ for the corresponding $V(T)\times U(T)$ submatrix of $B$. A subset $M\subset V(T)$ with  $|M|=|U(T)|$ is called a {\bf transversal} if there is a matching in $T$ that saturates $M\cup U(T)$. Since $T$ is a tree, every transversal $M$ has a unique matching saturating it, and therefore, we may identify a matching saturating $U(T)$ with its corresponding transversal. Hence,  $\det(D\res M)=\pm1$ if $M$ is a transversal and $0$ otherwise, where as usual for $K\subset V(T)$ we denote by $D\res K$ the $K\times U(T)$ submatrix of $D$. Cauchy-Binet formula now implies that  the number of transversals in $T$ is given by $\det(D^TD)$ and that the random subset $\cM$ of $V(T)$ sampled by the determinantal measure associated with the  image of $D$ is a {\em uniformly} drawn transversal.

Recall that for a subset $K\subset V(T)$ we write $m(K,T)$ for the number of matchings in $T$ in which $K\cup U(T)$ is saturated, or equivalently, the number of transversals containing $K$. 

\begin{claim}\label{clm:tree_matching_vector1} If $K\subset V(T)$ such $m(K,T)>0$ and $v_0 \in V(T)\setminus K$, then 
\[
m(K\cup\{v_0\},T)<m(K,T)\, .
\]
\end{claim}
\begin{proof}
If there is no matching saturating $K\cup\{v_0\}$, then the claim is trivial. Otherwise, given a matching $M$ in which the vertices $K\cup\{v_0\}$ are saturated, construct the following non-backtracking path in $T$ denoted $v_0,u_0,....,v_\ell,u_\ell,v_{\ell+1}$ for some $\ell \geq 0$ as follows. Start with $v_0$ and let $u_0\in U(T)$ be its unique neighbor in $M$. Next choose $v_1$ to be an arbitrary neighbor of $u_0$ in $T$ that is not $v_0$. If $v_1 \not \in K$, terminate. Otherwise, let $u_1$ be the unique neighbor of $v_1$ in $M$. Continue similarly until terminating. Since no vertex of $U(T)$ is a leaf of $T$ and since $T$ is a tree, the process must terminate at vertex $v_{\ell+1} \in V(T)\setminus (K \cup \{v_0\})$. 

Now, for each $0 \leq i \leq \ell$ replace the edge $\{v_i,u_i\}$ of $M$ with $\{v_{i+1},u_i\}$. This yields a matching that saturates $K$ but not $v_0$ and the claim follows.
\end{proof}

\begin{claim}\label{clm:tree_matching_vector2}
Let $V(T)=K\uplus J$ be a partition of $V(T)$ such that $m(K,T)>0$ and let $v_0\in J$. Then there exists a  vector $\varphi \in \R^{V(T)}$ such that 
\begin{enumerate}[(i)]
    \item $\varphi^T \cdot D=0$, and,
    \item $\varphi(v_0)=1$, and,
    \item $\sum_{v\in J} \varphi(v)^2 = \frac{m(K,T)}{m(K,T)-m(K\cup\{v_0\},T)}$. 
\end{enumerate}  
\end{claim}
\begin{proof}
We first claim that there exists a vector $\varphi$ satisfying (i) and (ii). Indeed, otherwise the row of $v_0$ in $D$ is {\emph not} spanned by the all the other rows. Since $\det(D\res M)=\pm1$ if $M$ is a transversal and $0$ otherwise, it follows the rows of $D$ corresponding to vertices of $V(T)$ saturated by $M$ are a basis of the row space of $D$; hence $v_0 \in \cM$ with probability $1$, contradicting \cref{clm:tree_matching_vector1}. 

Next, denote by $P_{(\R^K)^\perp}$ the orthogonal projection from $\R^{V(T)}$ onto $(\R^K)^\perp$, in other words, the operator zeroing the coordinates of $K$ and leaving the coordinates of $J$ unaltered. 
{
    By definition, the image under $P_{(\R^K)^\perp}$ of the subspace of vectors $\varphi$ satisfying condition (i) is precisely
    \[
    W:=(\ker D^T + \R^K)\cap(\R^K)^\perp\,.
    \]
    Moreover, if $w=P_{(\R^K)^\perp}\varphi$, then, since $v_0\in J$, we have $\varphi(v_0)=w(v_0)$ and
    \[
    \sum_{v\in J}\varphi(v)^2=\|w\|^2\,.
    \]
We  produce the desired vector $\varphi$ as a preimage of $w = P_W e_{v_0}/\|P_W e_{v_0}\|^2$ under $P_{(\R^K)^\perp}$.
Here, as usual, $P_W$ is the orthogonal projection onto $W$ and $e_{v_0}$ is the vector taking $1$ at $v_0$ and $0$ otherwise. 
In fact, this $w$ minimizes $\|w\|^2$ over the affine slice $W \cap \{ w\in \R^{V(T)} : w(v_0)=1\}$. Indeed, by letting $\alpha:=\|P_We_{v_0}\|^2$ we find that for $w$ in the affine slice 
    \[
    1-\alpha = \|e_{v_0}-P_W e_{v_0}\|^2 \le \|e_{v_0}-\alpha w\|^2=1-2\alpha+\alpha^2\|w\|^2\,,
    \]
whence $\|w\|^2\ge 1/\alpha$. To compute the value of $\alpha$ we observe that $$W^\perp= (\mathrm{im}(D)\cap(\R^K)^\perp)+\R^K$$ is the subspace of $\R^{V(T)}$ that induces the determinantal probability measure of $\cM$ conditioned on the event $K\subset \cM$, by \eqref{eq:conditioning}. Therefore,
\[
\alpha= 1-\|P^{W^\perp}e_{v_0}\|^2 =1-\P(v_0\in\cM\mid K\subset\cM) = 1-\frac{m(K\cup\{v_0\},T)}{m(K,T)}\,.
\]
In conclusion, if $\varphi \in \R^{V(T)}$  satisfies (i),(ii) and minimizes $\sum_{v\in J} \varphi(v)^2$ then
\[
\sum_{v\in J}\varphi(v)^2 = \frac{1}{\alpha} = \frac{m(K,T)}{m(K,T)-m(K\cup\{v_0\},T)} \, ,\]
as required.}
\end{proof}

The following easy claim deals with the case that $m(K,T)=0$.
\begin{claim} \label{clm:transversal0}
    Suppose that $m(K,T)=0$. Then, there exists a non-zero vector $\varphi \in \R^{V(T)}$ supported on $K$ such that $\varphi^T D=0$.
\end{claim}
\begin{proof} If $m(K,T)=0$, then the rows of $D$ corresponding to $K$ are linearly dependent since otherwise we could add more rows of $D$ and obtain a basis which corresponds to a transversal that contains $K$.
\end{proof}

\subsubsection{A determinantal version of the Nash-Williams inequality}
We will need to lower bound the probability that an element is chosen in a determinantal probability measure. The following is the analogue of the well-known Nash-Williams inequality lower bounding the effective resistance in an electric network (see \cref{rem:NW} below).
\begin{lemma}\label{lem:det_NW_Analogue} Let $\P^W$ be the determinantal probability measure associated with a subspace $W\subset \R^m$, denote by $\mathcal W\subseteq[m]$ a sample of $\P^W$ and fix $u\in [m]$. Suppose that $\xi_1,...,\xi_k$ are linearly independent vectors in $W$ such that $|\xi_i(u)|=1$ for all $i\in [k]$ and let $M$ be their invertible Gram matrix. Then 
$$ \P^W(u\in \cW) \ge \sum_{i,j=1}^k(M^{-1})_{i,j} \, .$$
\end{lemma}
\begin{proof} Assume without loss of generality that $\xi_i(u)=1$; otherwise replace $\xi_i$ by $-\xi_i$. By \eqref{eq:PHmarginal}, we have that $\P^W(u\in\mathcal W)=\|P_W e_u\|^2$ (as usual $e_u \in \R^m$ is the vector taking $1$ at $u$ and $0$ elsewhere) and so it is bounded from below by the squared norm of the projection of $e_u$ onto the linear span of $\xi_1,...,\xi_k$ since these belong to $W$. Let $Q$ be the $k\times m$ matrix whose rows are $\xi_1,...,\xi_k$ so that {the projection matrix onto the linear span of $\xi_1,...,\xi_k$ is $Q^T(QQ^T)^{-1}Q = Q^TM^{-1}Q$, see \cite[Chapter 12]{MatrixBook2}.} Our assumption that $\xi_i(u)=1$ implies that $e_u Q^T \in \R^{k}$ is the all-ones vector. Hence 
$$\P^W(u\in \cW) \ge e_u Q^TM^{-1}Qe_u^T = \sum_{i,j=1}^k(M^{-1})_{i,j} \, . $$
\end{proof}

The following corollary will be useful. 
\begin{corollary}\label{lem:det_NW}
Let $\P^W$ be the determinantal probability measure associated with a subspace $W\subset \R^m$, and  $\mathcal W\subseteq[m]$ be a sample of $\P^W$. Let $u\in [m]$ and $\gamma \in(0,{1 \over k})$ be fixed, where $k \in [m]$. Suppose that $\xi_1,...,\xi_k \in W$ are vectors satisfying
\begin{enumerate}[(i)]
\item $|\xi_i(u)|=1$ for all $i\in\{1,\ldots,k\}$, and,
\item $|\langle \xi_i,\xi_j\rangle|\le\gamma \min(\|\xi_i\|^2,\|\xi_j\|^2)$ for all $i \neq j \in \{1,\ldots,k\}$. 
\end{enumerate} 
Then
\[
\P^W(u\in\mathcal W) \ge \left(1-\frac{\gamma k}{1-\gamma k}\right)\cdot \sum_{i=1}^k\frac{1}{\|\xi_i\|^2}\,.
\]
\end{corollary}
\begin{proof} Let $M$ be the Gram matrix of $\xi_1,...,\xi_k$, denote by $D$ the diagonal matrix with entries $D_{i,i}=M_{i,i}=\|\xi_i\|^2$ and put $R=I-D^{-1}M$. The second assumption implies that  $|R_{i,j}|\le \gamma$ for all $i, j\in [k]$ (since $R_{i,i}=0$). By induction it follows that for any integer $\ell \geq 1$ we have $|(R^\ell)_{i,j}|\leq \gamma^\ell k^{\ell-1}$ since 
$$|(R^\ell)_{i,j}| \le \sum_{t=1}^k|R_{i,t}||(R^{\ell-1})_{t,j}| \, .$$
Therefore $\tilde R:= \sum_{\ell=1}^\infty R^\ell$ converges in every coordinate and the upper bound above implies that $|\tilde R_{i,j}| \leq {\gamma \over 1- \gamma k}$. The matrix $I+\tilde R$ is the inverse of $I-R=D^{-1}M$, hence $M$ is invertible, so $\xi_1,...,\xi_k$ are independent and $M^{-1}=(I+\tilde R)D^{-1}$. 
By \cref{lem:det_NW_Analogue} $$\P^W(u\in\cW)\ge \sum_{i=1}^k\left(1 + \sum_{j=1}^k\tilde R_{j,i}\right)\cdot \frac 1{M_{i,i}}\,,$$
and the proof is concluded since 
$
\tilde R_{i,j}\ge -{\gamma \over 1- \gamma k}$.
\end{proof}

\begin{remark}\label{rem:NW}
Let us explain why \cref{lem:det_NW_Analogue} extends the classical Nash-Williams inequality from electric networks theory (see \cite[Chapter 2]{LP16}) stating that if $x\neq y$ are two vertices of a graph $G$ and $C_1,\ldots,C_k$ are disjoint edge sets each separating $x$ from $y$, then the effective resistance from $x$ to $y$ is bounded below by $\sum_{i=1}^k {1 \over |C_i|}$. Using Kirchoff's theorem and the parallel law of electric networks (see \cite[Chapter 2]{LP16} and \cite[Chapter 4]{LP16}), it is easy to see that this is equivalent to the following statement about a uniform spanning tree $\cT$ of a finite connected graph $G$: if $e=(x,y)$ is an edge of $G$ and $C_1,\ldots, C_k$ are edge sets each separating $x$ from $y$ so that $C_i \cap C_j= \{e\}$ for all $i\neq j$, then,

\be\label{eq:nash_williams}
    \P(e\in\cT)\ge \left(1+\frac{1}{S}\right)^{-1}  \quad \text{where} \quad S=\sum_{i=1}^k\frac{1}{|C_i|-1} \, ,
\ee
where we interpret $1/S=0$ in the case that $C_i=\{e\}$ for some $i$ (in this case $G \setminus \{e\}$ is disconnected and $e \in \cT$ with probability $1$).


{
    Let us show how to prove \eqref{eq:nash_williams} using the method of \cref{lem:det_NW}. 
    In this case the corresponding subspace $W\subset\R^{|E|}$ is described in \cref{example:ust}. 
    We may assume that the edge sets $C_1,\ldots,C_k$ above \eqref{eq:nash_williams} are inclusion-wise minimal, and hence each $C_i$ is a cutset separating $x$ from $y$. 
    Orient the edges of $C_i$ towards the side of the cut containing $y$, and take $\xi_i$ to be the signed indicator function of $C_i$. 
    Then $\xi_i\in W$, $\|\xi_i\|^2=|C_i|$ and $\xi_i(e)=1$ for all $i$. 
    Furthermore, $\langle\xi_i,\xi_j\rangle=\xi_i(e)\xi_j(e)=1$ for every $i\neq j$ since $C_i \cap C_j = \{e\}$. 
    Let $M$ be the Gram matrix of $\xi_1,...,\xi_k$ and observe that $M=D+\mathbf 1\cdot\mathbf 1^T$ where $\mathbf 1 \in \R^k$ is the all-ones vector and $D$ is the diagonal matrix with entries $|C_1|-1,....,|C_k|-1$. By \cref{lem:det_NW_Analogue} we obtain that
    \[
    \P(e\in\cT)\ge \sum_{i,j=0}^k(M^{-1})_{i,j}=\mathbf 1^TM^{-1}\mathbf1 =\left(1+\frac{1}{\mathbf1^TD^{-1}\mathbf 1}\right)^{-1}\,,
    \]
    where the last transition follows directly from the Sherman–Morrison formula (see \cite[Section 0.7.4]{MatrixBook}). This recovers \eqref{eq:nash_williams} since $\mathbf1^TD^{-1}\mathbf 1=\sum_{i=1}^k 1/(|C_i|-1)=S.$
    }

\end{remark}

\subsubsection{Proof of \cref{lem:exactly}}

Denote by $N_G(v)$ the neighbor set of a vertex $v$ in the graph $G$ and let $(T,o)$ be a valid rooted tree.  
Conditioned on $U(T)\subset\cT$, the event that $B_\cT(o,r)=T$ occurs if and only if for all $v\in I(T)$ we have $(N_G(v)\setminus U(T)) \cap \cT = \emptyset$. 

Let $U_0=\{u\in U(G)\setminus U(T) \mid |N_G(u)\cap V(T)|>1\}$, and note that $|U_0|\le |V(T)|^2$ since every two vertices in $V(G)$ have at most one common neighbor in $U(G)$ (since $G$ is $C_4$-free). We write $I(T)=\{v_1,...,v_{|I(T)|}\}$ and for each $i \in \{1,\ldots, |I(T)|\}$ let $u_{i,1},....,u_{i,d_i}$ be the set of vertices $N_G(v_i)\setminus (U(T)\cup U_0)$, so $d\ge d_i \ge d - |V(T)|$ since $v_i$ can share at most one neighbor with each vertex in $V(T)$.

Additionally, for every $1\le i \le |I(T)|$ and $1\le j \le d_i$ let
\[
U_{i,j} = \bigcup_{i' < i} \Big \{u_{i', j'} : i' < i, j'\in \{1,\ldots,d_{i'}\} \Big \} \bigcup \Big \{ u_{i,j'} : j'<j \Big \} \, ,
\] 
that is, $U_{i,j}$ are the vertices which precede $u_{i,j}$ in lexicographical order. 
{By the observation above,
\[
\P(B_\cT (o,r)= T  \mid U(T) \subset \cT) \le
\P\left(\left(N_G(I(T))\setminus (U(T)\cup U_0)\right)\cap\cT=\emptyset\mid U(T) \subset \cT\right).
\]
Applying the chain rule to the vertices $u_{i,j}$ in lexicographical order gives}
\begin{equation}\label{eq:prob_chain_bound}
\P(B_\cT (o,r)= T  \mid U(T) \subset \cT) \le 
\prod_{i=1}^{|I(T)|}\prod_{j=1}^{d_i}
\P(u_{i,j} \notin \cT\mid U(T) \subset \cT,~U_{i,j}\cap\cT=\emptyset)\,.
\end{equation}

We now upper bound the probabilities in the right-hand side of \eqref{eq:prob_chain_bound} using \cref{lem:det_NW}.
For every $1\le i \le |I(T)|$, we denote $K_i=\{v_1,...,v_{i-1}\}\subset I(T)$ and 
\[
\beta_i:=\frac{m(K_i,T)}{m(K_i,T)-m(K_{i+1},T)}\,.
\]
Note that $\beta_i$ is well defined by \cref{clm:tree_matching_vector1} and that $\beta_i>1$. The following lemma upper bounds each term in \eqref{eq:prob_chain_bound}.
\begin{lemma}\label{lem:using_trans_NW}

For every $1\le i \le |I(T)|$ and $1\le j \le d_i$ we have
\[
\P(u_{i,j} \in \cT\mid U(T) \subset \cT,~U_{i,j}\cap\cT=\emptyset)\ge \frac{1+O(d^{-1})}{d}\left(k+ 
\left(
\beta_i
-\frac jd
 \right)^{-1} \right)\,.
\]
\end{lemma}
\begin{proof}
Denote by $P:\R^U(G)\to\R^{U\setminus U_{i,j}}$ denote the orthogonal projection (i.e., coordinate restriction) onto the complement of $U_{i,j}$ in $U(G)$. Recall that $H\subset \R^{U(G)}$ denotes the row-space of the matrix $B$ and consider its subspace
\[
H'= \{\psi\in H\mid \psi(u)=0\mbox{ for all $u\in U(T)$}\}\,.
\]
By \eqref{eq:conditioning}, the conditional probability measure given $U(T) \subset  \cT,~U_{i,j}\cap\cT=\emptyset$
considered in this lemma is the determinantal probability measure associated with the subspace 
\[
W = \{ P\psi\mid\psi\in H'\}\, .
\]
In  words, to construct a vector $\xi\in W$, one starts with a vector $\psi\in H$ that vanishes on $U(T)$ (i.e., $\psi\in H'$) and restricts its coordinates to $U\setminus U_{i,j}$.

Let $N_G(u_{i,j})=\{v_i,w_1,...,w_k\}$ be the neighbors of $u_{i,j}$ in $G$. Since $u_{i,j} \not \in U_0$ it follows that $\{w_1,\ldots,w_k\} \not \in V(T)$. We aim to use \cref{lem:det_NW} so we now construct $k+1$ vectors $\xi_0,...,\xi_k\in W$, where each vector is associated to a different neighbor of $u_{i,j}$.

Firstly, for every $1\le t \le k$, let $\psi_t$ be the $w_t$-th row of $B$. If $w_t$ has a neighbor in $U(T)$, then $w_t$ must belong to $V(T)$ since $T$ is valid. Since $w_t \not \in V(T)$ we deduce that $\psi_t\in H'$. 

Secondly, apply \cref{clm:tree_matching_vector2} with $K=K_i$ and $J=V(T)\setminus K_i$ and the vertex $v_i \in J$. Note that this partition satisfies the assumption in \cref{clm:tree_matching_vector2} since $m(K_i,T)\ge m(I(T),T)>0$. Denote by $\varphi \in \R^{V(T)}$  the vector obtained by \cref{clm:tree_matching_vector2} this way. Namely, we have that (i) $\varphi^T D=0$, (ii) $\varphi(v_i)=1$, and (iii) $\sum_{v\in J}\varphi(v)^2=\beta_i.$ Here $D$ is the $V(T)\times U(T)$ submatrix of $B$. We stress that $\varphi$ is determined only by the finite tree $T$ and the set $K_i$, and in particular it does not depend on $n$.
Extend $\varphi$ to $\tilde\varphi \in \R^V$ by padding zeroes in $V(G)\setminus V(T)$ and define $\psi_0:=\tilde \varphi ^T\cdot B$. Clearly, $\psi_0$ is in the row-space $H$ of $B$, and since $\varphi^T\cdot D=0$ it follows that $\psi_0\in H'$.

We define $\xi_t:=P\psi_t$ for $t=0,...,k$ and claim the following:

\begin{enumerate}
    \item $|\xi_t(u_{i,j})|=1$ for $0\le t\le k.$ \label{itm:1}
    \item $\|\xi_t\|^2=(1+O(d^{-1}))d$ for $1\le t\le k.$\label{itm:norm_xit}
    \item $\|\xi_0\|^2=(1+O(d^{-1}))(d\beta_i-j).$
    \label{itm:norm_xi0}
    \item $|\langle\xi_t,\xi_s\rangle|= O(1)$ for every $0\le t<s\le k$.
    \label{itm:inner}
\end{enumerate}
By combining these with the facts that $\beta_i>1$ and $j\le d$, we find that the squared norms of $\xi_0,...,\xi_t$ are of order $d$ while the inner-products are of order $1$. Therefore, by \cref{lem:det_NW} it suffices to prove the above four items. 

Firstly, by definition for $1\le t \le k$ we have  $\xi_t(u_{i,j})=B(w_t,u_{i,j})=\pm1$. Also, for $t=0,$
    \[
    \xi_0(u_{i,j}) = \psi_0(u_{i,j}) = (\tilde \varphi^T B)(u_{i,j})=\varphi(v_i)\cdot B(v_i,u_{i,j})=\pm 1\,,
    \]
    where we used that $\tilde \varphi(w_t)=0$ for $1\le t\le k$ since $w_t\notin V(T)$ and that $\tilde\varphi(v_i) = \varphi(v_i)=1$.

Secondly, for $1\le t\le k$, we have 
\[
\|\xi_t\|^2 = \bigl|N_G(w_t)\setminus U_{i,j}\bigr|,
\] 
which is in $[d-|V(T)|,d]$ since $w_t$ can share at most one neighbor with each vertex in $V(T)$ (so since $U_{i,j}$ are all neighbors of $\{v_1,\ldots, v_i\}$ the number of neighbors $w_t$ has in $U_{i,j}$ is at most $|V(T)|$).


Thirdly, let $u \in U(G)\setminus (U_{i,j} \cup U_0)$. Since $\tilde \varphi$ is supported on $V(T)$ and $u$ has no neighbors in $K_i$ (since $u \not\in U_{i,j}$) it follows that if $u$ has no neighbors in $J$, then $\xi_0(u)=\psi_0(u)=0$. Also, since $u\notin U_0$, it can have more than one neighbor in $J$ only if it belongs to $U(T)$, in which case
    \[
    \xi_0(u) = (\tilde\varphi^TB)(u)=(\varphi^T D)(u)=0\, ,
    \]
    where the second equality follows since all the neighbors of $u$ belong to $V(T)$, and the last equality by the construction of $\varphi$. Therefore $\xi_0(u)\ne 0$ implies that it has one neighbor $v\in J$ and hence $\xi_0(u)=\varphi(v)\cdot B(v,u)=\pm\varphi(v)$. This gives that 
    \[
    \|\xi_0\|^2 = \sum_{v\in J}\varphi(v)^2\cdot |N_G(v)\setminus (U(T)\cup U_{i,j})|+\sum_{u\in U_0}\xi_0(u)^2\,.
    \]
    Observe that the second sum bounded by $|U_0|(k+1)^2\|\varphi\|_\infty^2=O(1)$ and that
    \[
    |N_G(v)\setminus (U(T)\cup U_{i,j})| =
    \left\{
    \begin{matrix}
        d-j-O(1)&v=v_i \\
        d-O(1)& v\in J\setminus \{v_i\}\,.
    \end{matrix}
    \right.
    \]
    Indeed, the $O(1)$ terms account for neighbors of $v$ in $U_0\cup U(T)$. Besides these neighbors, $U_{i,j}$ contains none of the neighbors of $v$ if $v\ne v_i$ and $j-1$ of them if $v=v_i$. We derive that

    \[
    \|\xi_0\|^2 = \sum_{v\in J\setminus\{v_i\}}\varphi(v)^2(d-O(1)) + \varphi(v_i)^2\cdot (d-j-O(1))+O(1)=(1+O(d^{-1}))(d\beta_i-j)\,,
    \]
where the last equality follows from $\sum_{v\in J}\varphi(v)^2=\beta_i$ and $\varphi(v_i)=1$.

Lastly, if $1\le t<s\le k$ then 
\(
|\langle \xi_t,\xi_s\rangle| = 1
\) since $w_t$ and $w_s$ share exactly one common neighbor, namely $u_{i,j}$. In addition, if $1\le t\le k$ then 
\[
|\langle \xi_0,\xi_t\rangle| \le \sum_{u\in N_G(w_t)\setminus U_{i,j}}|\xi_0(u)|\, ,
\]
since $\xi_t$ is supported on $N_G(w_t)\setminus U_{i,j}$ and has sup norm at most $1$. 
Recall that if $u\notin U_{i,j}\cup U_0$ then $\xi_0(u)\ne 0$ only if $u$ has a unique neighbor $v\in J$, and then $|\xi_0(u)|=|\varphi(v)|.$ Since $w_t$ has at most one common neighbor with every $v\in J$, we find that
\[
|\langle \xi_0,\xi_t\rangle| \le \sum_{v\in J}|\varphi(v)|+\sum_{u\in U_0}|\xi_0(u)| = O(1) \, ,
\]
since $|J|=O(1)$ and $\|\varphi\|_\infty = O(1)$ and  $\|\xi_0\|_\infty\leq (k+1)\|\varphi\|_\infty=O(1)$ and $|U_0|=O(1)$. 
\end{proof}

To conclude the proof of \cref{lem:exactly}, we combine \cref{lem:using_trans_NW}, \eqref{eq:prob_chain_bound} and $1-x \leq e^{-x}$ to obtain
\begin{align*}
\P(B_\cT (o,r)= T  \mid U(T) \subset \cT) &\le 
\exp\left(-
(1+O(d^{-1}))\sum_{i=1}^{|I(T)|}\frac 1d\sum_{j=1}^{d_i} \left(k+ 
\left(\beta_i-\frac jd \right)^{-1} \right)\right)\\
&=
\exp\left(-
(1+O(d^{-1}))\sum_{i=1}^{|I(T)|} k+\log\left(\frac{\beta_i}{\beta_i-1}\right)\right)\,,
\end{align*}
where the last equality follows from  the approximation
$$\frac 1d\sum_{j=1}^d(\beta-j/d)^{-1}=(1+O(d^{-1}))\log\left(\frac{\beta}{\beta-1}\right)$$ for every real $\beta>1$, and from $d_i=d-O(1)$, which implies that
any errors can be absorbed into the $O(d^{-1})$ error term. Note that $\beta_i/(\beta_i-1) = m(K_i,T)/m(K_{i+1},T)$ whence
\[
\sum_{i=1}^{|I(T)|} \log\left(\frac{\beta_i}{\beta_i-1}\right) = \log \left(\frac{m(\emptyset,T)}{m(I(T),T)}\right)\,. 
\]
Therefore,
\[
\P(B_\cT (o,r)= T  \mid U(T) \subset \cT) \le 
(1+O(d^{-1}))\frac{e^{-k|I(T)|} m(I(T),T)}{m(\emptyset,T)}
\]
which concludes the proof of \cref{lem:exactly}. \qed

\subsubsection{Proof of \cref{lem:no_matching}}
Let $K \subset V(T)$ denote the set of even vertices in $T$ of height smaller than $r$. Suppose, for contradiction, that $T$ is of height $r'<r$, or that it is of height $r$ and $m(I(T),T)=0$. 
We claim that this implies that $m(K,T)=0$. Indeed, in the former case, $K=V(T)$ and since $|V(T)|>|U(T)|$ we have that $m(K,T)=0$. In the latter case $K=I(T)$ so $m(K,T)=0$ is assumed to hold.

Let $\varphi\in \R^{V(T)}$ be the vector obtained by applying \cref{clm:transversal0} on $T$ with $K$. That is, $\varphi:V(T)\to \R$ is a non-zero vector such that $\varphi^TD=0$ where $D$ is the restriction of $B$ to $V(T)\times U(T)$. Extend $\varphi$ to $\tilde\varphi \in \R^V$ padding zeros in $V(G)\setminus V(T)$.  

Consider a vertex $v\in K$ with $\varphi(v)\ne 0$. There exists $u\in U(G)$ such that $v$ is its only neighbor in $K$. Indeed, since $G$ is $C_4$-free, only at most $|K|$ of the $d$ neighbors of $v$ have an additional neighbor in $K$, and we assume $|K|\le |V(T)|<d$. Therefore, $(\tilde\varphi^T B)(u)=\pm\varphi(v)\ne0$. On the other hand, for any $u\in\cT$, if $u$ does not have a neighbor in $K$, then $(\tilde\varphi^T B)(u)=0$ since $\tilde\varphi$ is zero outside of $K$; if $u$ does have a neighbor in $K$ it follows that $u\in U(T)$ by definition of $K$ and since  $B_\cT(o,r)\cong (T,o)$, hence $(\tilde\varphi^T B)(u)=(\varphi^T D)(u)=0$. 

We deduce that $\tilde\varphi^T B\ne 0$ but $(\tilde\varphi^T B)(u)=0$ for every $u\in\cT$; consequently the vector $\tilde\varphi^T B$ is orthogonal to every column of $\cT$ but there is a column of $B$ that it is not orthogonal to. It follows that $\cT$'s columns do not span the column-space of $B$ and hence $\P^H(\cT)=0$, in contradiction.    
\qed

\subsection{Proof of \cref{thm:mainTree}}

Recall the parameters $\eps, \delta$ defined in \eqref{eq:eps_del} and the subset $U_{\eps,\delta}$ defined in \cref{def:struct}. By combining \cref{lem:inclusion} and \cref{lem:exactly} we obtain the following.

\begin{corollary}\label{cor:mainUpper} Let $(T,o)$ be a valid rooted finite tree which is a subgraph of $G$ and $o\in V$. Suppose that $U(T) \cap U_{\eps, \delta}=\emptyset$ and that $m(I(T),T)>0$, then
\[
\P(B_\cT (o,r)= (T,o)) \le (1+o(1)) e^{-k |I(T)|}m(I(T),T) d^{-|U(T)|} \, .
\]    
\end{corollary}

Therefore, to prove \cref{thm:mainTree} we need to count the number of finite rooted trees that are subgraphs of $G$ and are isomorphic as rooted graphs to a given abstract finite rooted trees $(T,o)$. The following counting argument is due to M{\'e}sz{\'a}ros \cite[Lemma 3.10]{meszaros2021local}; we provide its proof in our setup here (in \cite{meszaros2021local} the setup is for determinantal hypertrees \cref{sec:hypertrees} of the complete complex only). Given a finite rooted tree $(T,o)\in \cG_\bullet$ and a vertex $v\in V$ of the underlying graph $(V,U,E)$ we define
$$ \cI( (T,o); v) = \Big \{ T'\mid (T',v) \cong (T,o), \, (T',v) \textrm{ is a subgraph of } G \Big \} \, .$$

\begin{lemma} \label{lem:count} For any finite rooted tree $(T,o)\in \cG_\bullet$ and any $v\in V$ we have 

$$ {|\cI( (T,o); v)| \over d^{|U(T)|}} \leq  { (k!)^{|U(T)|} \over  |\mathrm{Aut}(T,o)| } \, .
$$
\end{lemma}
\begin{proof} Consider a fixed arbitrary labeling of the vertices of $T$ and an embedding $m:V(T)\cup U(T) \to V(G) \cup U(G)$ of $(T,o)$ into $G$ so that $m(o)=v$; by embedding we mean that $m$ preserves adjacency and non-adjacency. The number of such $m$'s is thus precisely $|\mathrm{Aut}(T,o)||\cI( (T,o); v)|$. This number is can be easily upper bounded inductively. The root $o$ must be mapped to $v$. Next, if the root $o$ has degree $r$, there are at most $d^r$ ways of embedding the neighbors of $o$ into vertices $u_1,\ldots, u_r$ of $U(G)$ . Next, each of the $r$ neighbors of $o$ in $T$ have precisely $k$ children so the embedding $m$ must embed these children into each of the $k$ neighbors of $u_1,\ldots, u_r$ that are not $v$ (note that we do not rule out the invalid possibility that these $r$ sets of $k$ neighbors intersect; this is fine since we are only proving an upper bound) and due to the labeling we have $(k!)^r$ such possibilities. We proceed iteratively and obtain an upper bound of precisely $d^{|U(T)|} (k!)^{|U(T)|}$.
\end{proof}

\begin{proof}[Proof of \cref{thm:mainTree}]

Let $v_n$ be a uniformly drawn vertex of $V$ that is independent from $\cT$. The third item of \cref{thm:TisTight} together with \cref{cor:mainUpper} and \cref{lem:count} as well as \cref{lem:no_matching} yields that
\be\label{eq:limsup}
\P(B_{\cT_n}(v_n,r)\cong (T,o)) \leq \frac{e^{-k|I(T)|}(k!)^{|U(T)|} m(I(T),T)}{|\mathrm{Aut}(T,o)|} + o(1) \, ,
\ee
for any valid finite rooted tree $(T,o)$ of height $r$, where the $o(1)$ term  may depend on $(T,o)$. Let $\eps>0$ be arbitrary.

The first and second items of \cref{thm:TisTight} and \cref{lem:no_matching} imply that there exists $t_1 = t_1(\eps) \geq 0$ such that for all $n$ large enough
\be\label{eq:tocontradict} \sum_{(T,o) : |T|\leq t_1} \P(B_{\cT_n}(v_n,r)\cong (T,o)) \geq  1 - \eps \, ,\ee
where the sum is over all valid finite rooted trees of height $r$. Furthermore, since $B_{\T_k}(\rho,r)$ is supported in valid finite rooted trees of height $r$, 
\eqref{eq:meszaros1.4} implies that 
$$ \sum _{(T,o)}  \frac{e^{-k|I(T)|}(k!)^{|U(T)|} m(I(T),T)}{|\mathrm{Aut}(T,o)|} = 1 \, ,$$ 
so in particular there exists $t_2 = t_2(\eps)\geq 0$ such that
\be\label{eq:finiteSum}
\sum _{(T,o) : |T|\leq t_2}  \frac{e^{-k|I(T)|}(k!)^{|U(T)|} m(I(T),T)}{|\mathrm{Aut}(T,o)|} \in [1-\eps,1] \, ,
\ee
where again the sum is over all valid finite rooted tree of height $r$.
Therefore, if there exists a valid rooted tree $(T,o)$ of height $r$ with $|T|\leq \max(t_1,t_2)$ for which $m(I(T),T)>0$ and such that
$$  \P(B_{\cT_{n}}(v_{n},r)\cong (T,o)) \leq \frac{e^{-k|I(T)|}(k!)^{|U(T)|} m(I(T),T)}{|\mathrm{Aut}(T,o)|} - 2\eps +o(1) \, ,$$
for all large enough $n$,
then putting this with \eqref{eq:limsup} and \eqref{eq:finiteSum} gives a contradiction to \eqref{eq:tocontradict}. This concludes the proof.
\end{proof}

\subsection{Proof of \cref{thm:quenched}}
\cref{thm:mainTree} readily implies that $$\lim_{n\to\infty}\E\left[\frac{Y_n}{|V_n|}\right]=  \frac{e^{-k|I(T)|}(k!)^{|U(T)|} m(I(T),T)}{|\mathrm{Aut}(T,o)|}\,,$$ 
so it remain to show that $\lim_{n\to\infty}\mbox{Var}(Y_n/|V_n|)=0$. For this purpose, we need the following extension of \cref{cor:mainUpper}.

\begin{lemma}\label{lem:two_trees}
Let $(T,o)$ and $(T',o')$ be isomorphic valid rooted finite trees which are subgraphs of $G$ and $o,o'\in V$. Suppose that $U(T)\cup V(T),U(T')\cup V(T')$ and $U_{\eps, \delta}$ are pairwise disjoint
and that $m(I(T),T)>0$. Then,
\[
\P(B_\cT (o,r)= (T,o),~B_\cT (o',r)= (T',o')) \le (1+o(1)) \left(e^{-k |I(T)|}m(I(T),T) d^{-|U(T)|} \right)^2\, .
\]    
\end{lemma}

\begin{proof} By \cref{lem:inclusion} and negative correlation \eqref{eq:negative} we have that
\[
\P(U(T)\cup U(T')\subset\cT)\le 
(1+o(1)) \left(d^{-|U(T)|} m(\emptyset, T)\right)^2 \, .
\]
Therefore, it remains to show that
\begin{align}
\nonumber \P(B_\cT (o,r)= (T,o),~
B_\cT (o',r)&= (T',o')
\mid U(T)\cup U(T') \subset \cT) \\ &\le (1+o(1))\left( \frac{e^{-k |I(T)|}m(I(T),T)}{m(\emptyset, T)}\right)^2\, .    \label{eq:double_exclusion}
\end{align}

The proof of \eqref{eq:double_exclusion} follows the same lines of the proof of \cref{lem:exactly}. Let $U_0=\{u\in U(G)\setminus (U(T)\cup U(T'))\mid |N_G(u)\cap (V(T)\cup V(T'))|>1\}\,,$ and note that $|U_0|=O(1)$ since $G$ is $C_4$-free. 
Denote $I(T)=\{v_1,...,v_{I(T)}\}$ and $I(T')=\{v'_1,...,v'_{I(T)}\}$ so that $v'_i$ is the image of $v_i$ under some fixed rooted-tree isomorphism. To ease notations, let $\bar d=d-O(1)$ be an integer such that every $v\in I(T)\cup I(T')$ has at least $\bar d$ neighbors outside of $U_0\cup U(T)\cup U(T')$. Let $u^1_{i,1},...,u^1_{i,\bar d}$ be a set of such neighbors of $v_i$, and $u^2_{i,1},...,u^2_{i,\bar d}$ a set of such neighbors of $v'_i$. Note that these $2|I(T)|$ neighbor sets are pairwise disjoint by the definition of $U_0$. Denote 
\[
U_{a,i,j}= \{u^{a'}_{i',j'}\mid (a',i',j')<_\mathrm{lex} (a,i,j)\}\,,
\]
where $<_{\mathrm{lex}}$ denotes the lexicographic order.
We claim that for every $1\le a\le 2$, $1\le i\le |I(T)|$ and $1\le j\le \bar d$ we have
\begin{equation}\label{eq:uaij}
\P(u^a_{i,j} \in \cT\mid U(T) \cup U(T')\subset \cT,~U_{a,i,j}\cap\cT=\emptyset)\ge \frac{1+O(d^{-1})}{d}\left(k+ 
\left(
\beta_i
-\frac jd
 \right)^{-1} \right)\,,
\end{equation}
where
\[
\beta_i:=\frac{m(K_i,T)}{m(K_i,T)-m(K_{i+1},T)}\,.
\]
Note that \eqref{eq:double_exclusion} follows from \eqref{eq:uaij} by applying a similar computation as in the last part of the proof of \cref{lem:exactly}.

To derive \eqref{eq:uaij} we follow the same argument as in the proof  of \cref{lem:using_trans_NW}, using the fact that $T$ and $T'$ are vertex disjoint.
Let $H' = H \cap [U(T)\cup U(T')]^\perp$ and $W=\{P\psi\mid\psi \in H'\}$ where $P$ is the projection given by restricting  to the coordinates not in $U_{a,i,j}$. Our goal is to bound the projection of the unit vector of $u^a_{i,j}$ to $W$ from below. The neighbors of $u^a_{i,j}$ are $v\in \{v_i,v'_i\}$ and $w_1,...,w_k \notin V(T)\cup V(T')$. As before, we use the rows $\psi_1,...,\psi_k$ of $w_1,...,w_k$ in $B$, and the linear combination $\psi_0$ given by the vector $\varphi$ from \cref{clm:tree_matching_vector2} of the rows of $V(T)$ if $v=v_i$, or of the rows of $V(T')$ if $v=v'_i$. The former vectors are in $H'$ since $w_1,...,w_k$ are not in $V(T)\cup V(T')$. For $\psi_0$, suppose, wlog, that it is a linear combination of the rows of $V(T)$. Then, $\psi_0(u)=0$ for $u\in U(T)$ by item (i) in \cref{clm:tree_matching_vector2} as before, and $\psi_0(u)=0$ for $u\in U(T')$ since $V(T)\cap V(T')=\emptyset$ whence no vertex in $V(T)$ has a neighbor in $U(T')$. 
Therefore, $\psi_0$ is also in $H'$. From this point, we define $\xi_t=P\psi_t, ~t=0,...,k$ and proceed precisely as in the proof of \cref{lem:using_trans_NW}.
\end{proof}

\begin{proof}[Proof of \cref{thm:quenched}]
Let $v_n,v_n'$ denote two independent uniform vertices from $V_n$ which are independent from $\cT$. Denote by $E$ (resp. $E'$) the event that $B_\cT(v_n,r)\cong(T,o)$ (resp. $B_\cT(v_n',r)\cong(T,o)$). We have that
\(
\E\left[Y_n^2/|V_n|^2\right]= \P(E\cap E'),
\)
and we bound this probability by considering the following cases. First, by \cref{thm:TisTight}, the probability that one of the $r$-balls around $v_n$ or $v_n'$ touch $U_{\varepsilon,\delta}$ is negligible. Second, by \eqref{eq:same_balls} the probability that the balls of radii  $2r$ around $v_n$ in $G[\cF]$ and $G[\cT]$ are different is also negligible. Therefore
\[
\P(B_\cT(v_n,r) \cap B_\cT(v_n',r) \ne\emptyset) = \P(v_n'\in B_\cF(v_n,2r))+o(1) = |V_n|^{-1}\E[|B_\cF(v_n,2r)|]+o(1)\,
\]
which is negligible since the expectation of $|B_\cF(v_n,2r)|$ is bounded as $n\to \infty$ by \eqref{eq:B_F}. In conclusion, using \cref{lem:count} and \cref{lem:two_trees} we find that, as $n\to\infty$,
\[
\P(E\cap E') \le \left(\frac{e^{-k|I(T)|}(k!)^{|U(T)|} m(I(T),T)}{|\mathrm{Aut}(T,o)|}\right)^2+o(1)\,,
\]
whence $\mbox{Var}(Y_n/|V_n|)\to 0$ as $n\to\infty$, and \cref{thm:quenched} follows by Chebyshev's inequality.
\end{proof}

\section*{Acknowledgements} We thank Rick Kenyon for many enlightening relevant discussions. The first author is supported by ERC consolidator grant 101001124 (UniversalMap) as well as ISF grants 1294/19 and 898/23. 
The second author was partially supported by the ISF grant 3464/24. 

\bibliographystyle{abbrv}
\bibliography{det}
\end{document}